\documentclass[11pt]{article}

\usepackage{amsthm}
\usepackage{amssymb}
\usepackage{amsmath}
\usepackage{enumitem} %allows easier labelling/referencing with enumerate
\usepackage{hyperref}
\usepackage[capitalise,compress]{cleveref} %smart references, capitalise (e.g.) Theorems, ``Figure'' not ``Fig.''
\usepackage{comment}
\usepackage{fullpage}

\usepackage{thmtools}
\usepackage{thm-restate}

\usepackage{subcaption}
\usepackage{tkz-berge}
\usepackage{tikz}
\usetikzlibrary{snakes}

\usepackage{authblk}

\newtheorem{lemma}{Lemma}[section]
\newtheorem{theorem}[lemma]{Theorem}
\newtheorem{sublemma}{Claim}[lemma]
\newtheorem{corollary}[lemma]{Corollary}
\newtheorem{proposition}[lemma]{Proposition}
\newtheorem{question}[lemma]{Question}
\newtheorem{conjecture}[lemma]{Conjecture}
\theoremstyle{definition}

\crefname{sublemma}{Claim}{Claims}

\newcommand{\ba}{\backslash}
\newcommand{\seq}[1]{\{1,2,\dotsc,#1\}}

\newcommand{\lc}{\kappa} 

\newcommand{\fatras}{wheel morass} 

\newcommand\uthanks[1]{%
  \begingroup
  \renewcommand\footnotemark{}\footnote{#1}
  \endgroup
}

\title{Colouring graphs with constraints on connectivity%
  \uthanks{The first author was supported by Fondecyt Postdoctoral grant 3150314 of CONICYT Chile.
  The second, third, and fifth authors were partially supported by ANR project Stint under reference ANR-13-BS02-0007 operated by the French National Research Agency (ANR).
  The second and fifth authors were partially supported by ANR project Heredia under reference ANR-10-JCJC-0204, and by the LABEX MILYON (ANR-10-LABX-0070) of Universit\'e de Lyon, within the program ``Investissements d'Avenir'' (ANR-11-IDEX-0007) operated by the French National Research Agency (ANR).
  The fourth author was supported by the European Research Council (ERC)  grant ``PARAMTIGHT: Parameterized complexity and the search for tight complexity results,'' reference 280152 and OTKA grant NK105645.}
  \uthanks{\textit{Email addresses: }\texttt{pierreaboulker@gmail.com} (P.~Aboulker), \texttt{nbrettell@gmail.com} (N.~Brettell), \texttt{frederic.havet@cnrs.fr} (F.~Havet), \texttt{dmarx@cs.bme.hu} (D.~Marx), \texttt{nicolas.trotignon@ens-lyon.fr} (N.~Trotignon).}
}

\author[1]{Pierre Aboulker}
\affil[1]{Universidad Andres Bello, Santiago, Chile}
\author[2]{Nick Brettell}
\affil[2]{CNRS, LIP, ENS de Lyon}
\author[3]{Fr\'ed\'eric Havet}
\affil[3]{Project Coati, I3S (CNRS, UNS) and INRIA, Sophia Antipolis, France}
\author[4]{D\'aniel Marx}
\affil[4]{Institute for Computer Science and Control, Hungarian Academy of Sciences (MTA SZTAKI)}
\author[2]{Nicolas Trotignon}

\date{\today}

\begin{document}

\maketitle

\begin{abstract}
  A graph $G$ has \emph{maximal local edge-connectivity~$k$} if
the maximum number of edge-disjoint paths between every pair of distinct vertices $x$ and $y$ is at most $k$.
  We prove Brooks-type theorems for $k$-connected graphs with maximal local edge-connectivity~$k$, and for any graph with maximal local edge-connectivity~$3$.
  We also consider several related graph classes defined by constraints on connectivity.
  In particular, we show that there is
  a polynomial-time algorithm that, given a $3$-connected graph $G$ with maximal local connectivity~$3$, outputs an optimal colouring for $G$.
  On the other hand, 
  we prove, for $k \geq 3$, that \textsc{$k$-colourability} is NP-complete when restricted to minimally $k$-connected graphs,
  and \textsc{$3$-colourability} is NP-complete when restricted to $(k-1)$-connected graphs with maximal local connectivity~$k$.
  Finally, we consider a parameterization of \textsc{$k$-colourability} based on the number %$p$
  of vertices of degree at least $k+1$, and prove that, even when $k$ is part of the input, the corresponding parameterized problem is FPT.
\end{abstract}


%\subjclass{O5C15,05C40,05C85}

\section{Introduction}

%Let \textsc{$k$-colourability} be the problem that takes an instance of
%a graph $G$, and answers ``yes'' or ``no'' according to the
%existence of a proper vertex $k$-colouring for $G$.
%\textsc{$3$-colourability} is NP-complete even when restricted to planar graphs with vertices of degree at most $4$~\cite{Garey1976}. Moreover,
%for any fixed $k\geq 3$,
%\textsc{$k$-colourability} is NP-complete. %,

We consider the problem of finding a proper vertex $k$-colouring for a graph for which, loosely speaking, the ``connectivity'' is somehow constrained.
%One naive approach by which we can constrain the connectivity is to restrict the maximum degree of a graph.
For example, if we consider the class of graphs of degree at most $k$, then, by Brooks' theorem, it is easy to find if a graph in this class is $k$-colourable.

\begin{theorem}[Brooks, 1941]
  Let $G$ be a connected graph with maximum degree $k$.  Then $G$ is $k$-colourable if and only if $G$ is not a complete graph or an odd cycle.
\end{theorem}

\noindent
On the other hand, if we consider the class of graphs with maximum degree $4$, then
the decision problem \textsc{$3$-colourability} is well known to be NP-complete, even when restricted to planar graphs~\cite{Garey1976}. Moreover,
for any fixed $k\geq 3$,
\textsc{$k$-colourability} is NP-complete. 
%In what follows, we consider several graph classes that fall somewhere between these two ends of the spectrum.

The classes we consider are defined using the notion of local connectivity.
The \emph{local connectivity} $\lc(x, y)$ of distinct vertices $x$ and $y$ in a graph is the maximum number of internally vertex-disjoint paths between $x$ and $y$.
The \emph{local edge-connectivity} $\lambda(x, y)$ of distinct vertices $x$ and $y$ is the maximum number of edge-disjoint paths between $x$ and $y$.
%A graph $G$ is said to have \emph{maximal local edge-connectivity~$k$} if $\lambda(x, y) \leq k$ for all distinct vertices $x$ and $y$ in $G$.
%
%The class of graphs with maximal local edge-connectivity~$k$, denoted $\mathcal{C}_1^k$, is the smallest of three nested classes that we consider.
%As an application, we investigate computational aspects of the chromatic
%number of graphs defined by constraints on connectivity.  
%Intuitively, one might expect it to be easier to colour a graph that is ``less connected''.
%Towards this hypothesis, we consider four classes; but first we require some definitions.
%The classes are:
Consider the following classes:

\begin{itemize}
  \item ${\mathcal C}_0^k$: graphs with maximum degree~$k$,
  \item ${\mathcal C}_1^k$: graphs such that $\lambda(x, y) \leq k$ for all pairs of distinct vertices $x$ and $y$,
  \item ${\mathcal C}_2^k$: graphs such that $\lc(x, y) \leq k$ for all pairs of distinct vertices $x$ and $y$, and
  \item ${\mathcal C}_3^k$: graphs such that $\lc(x, y) \leq k$ for all edges $xy$. %, and
  %\item ${\mathcal C}_4^k$: graphs with no $(k+1)$-connected subgraph.
\end{itemize}

\noindent
In each successive class, the connectivity constraint is relaxed; that is,
$\mathcal{C}_{0}^k \subseteq \mathcal{C}_1^k \subseteq {\mathcal C}_2^k \subseteq {\mathcal C}_3^k$. % \subseteq {\cal C}_4^k$.
For each class, there is a bound on the chromatic number; we give details shortly.
Note also that each of the four classes is closed under taking subgraphs.

A graph $G$ is \emph{$k$-connected} if it has at least $2$ vertices and $\lc(x,y) \geq k$ for all distinct $x,y \in V(G)$.
The \emph{connectivity} of a graph $G$ is the maximum integer $k$ such that $G$ is $k$-connected.
A graph contained in one of the above classes has connectivity at most $k$.
So, for each class, it may be of interest to %study
start by considering
the graphs that have connectivity precisely $k$.
For each class $\mathcal{C}_i^k$, we denote by $\widehat{\mathcal{C}}_i^k$ the subclass containing the $k$-connected members of $\mathcal{C}_i^k$.
A Hasse diagram illustrating the partial ordering of these classes under set inclusion is given in \cref{hassepic}.

\begin{figure}
  \centering
  \begin{tikzpicture}
    %\node (c3k) at (0,4) {$\mathcal{C}_4^k$};
    %\node (ch3k) at (1.5,3) {$\widehat{\mathcal{C}}_4^k$};
    \node (c2k) at (0,2) {$\mathcal{C}_3^k$};
    \node (ch2k) at (1.5,1) {$\widehat{\mathcal{C}}_3^k$};
    \node (c1k) at (0,0) {$\mathcal{C}_2^k$};
    \node (ch1k) at (1.5,-1) {$\widehat{\mathcal{C}}_2^k$};
    \node (c0k) at (0,-2) {$\mathcal{C}_1^k$};
    \node (ch0k) at (1.5,-3) {$\widehat{\mathcal{C}}_1^k$};
    \node (cn1k) at (0,-4) {$\mathcal{C}_{0}^k$};
    \node (chn1k) at (1.5,-5) {$\widehat{\mathcal{C}}_{0}^k$};
    \draw (c2k) -- (c1k) -- (c0k) -- (cn1k) -- (chn1k)
    (c2k) -- (ch2k) -- (ch1k) -- (ch0k) -- (chn1k)
    (c1k) -- (ch1k)
    (c0k) -- (ch0k);
  \end{tikzpicture}
  \caption{Hasse diagram of the graph classes defined by constraints on connectivity under $\subseteq$.}
  \label{hassepic}
\end{figure}
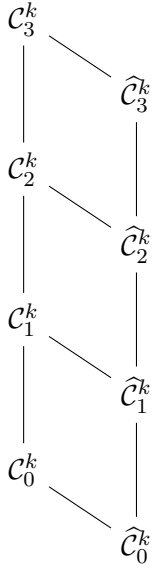

A graph in $\mathcal{C}_1^k$ is said to have \emph{maximal local edge-connectivity~$k$}. % if $\lambda(x, y) \leq k$ for all distinct vertices $x$ and $y$ in $G$.
Our first main result is
a Brooks-type theorem for graphs with maximal local edge-connectivity~$k$.
%We prove the following:
An \emph{odd wheel} is a graph obtained from a cycle of odd length by adding a vertex that is adjacent to every vertex of the cycle.

\begin{theorem}
  \label{main}
  Let $G$ be a $k$-connected graph with maximal local edge-connectivity~$k$, for $k \geq 3$.  Then $G$ is $k$-colourable if and only if $G$ is not a complete graph or an odd wheel.
\end{theorem}

\noindent
Note that an odd wheel is not $4$-connected, so the condition that $G$ is not an odd wheel is only required when $k=3$. %, when $k \geq 4$, the graph $G$ is $k$-colourable if and only if $G$ is not a complete graph.

Although every graph with maximum degree $k$ has maximal local edge-connectivity~$k$, % at most~$k$,
\cref{main}
is not, strictly speaking, a generalisation of Brooks' theorem, since it
only concerns such graphs that are $k$-connected.
However, for $k=3$ we prove an extension of Brooks' theorem that characterises which graphs with maximal local edge-connectivity~$3$ are $3$-colourable,
with no requirement on $3$-connectivity.

%Let $G_1$ and $G_2$ be graphs.  For $i \in \{1,2\}$,
%let $p_i$ be an edge of $G_i$ together with an orientation for that edge: for the edge $p_i$, let $u_i$ be the head and let $v_i$ be the tail.
%We say that the \emph{Haj\'os join} of $G_1$ and $G_2$ with respect to $p_1$ and $p_2$ is the graph obtained by deleting $p_1$ and $p_2$ from $G_1$ and $G_2$, respectively, identifying $u_1$ and $u_2$, and adding a new edge joining $v_1$ and $v_2$.
Let $G_1$ and $G_2$ be graphs and, for $i \in \{1,2\}$,
let $(u_i,v_i)$ be an ordered pair of adjacent vertices of $G_i$. 
We say that the \emph{Haj\'os join} of $G_1$ and $G_2$ with respect to $(u_1,v_1)$ and $(u_2,v_2)$ is the graph obtained by deleting the edges $u_1v_1$ and $u_2v_2$ from $G_1$ and $G_2$, respectively, identifying the vertices $u_1$ and $u_2$, and adding a new edge joining $v_1$ and $v_2$.
A \emph{block} of a graph $G$ is a maximal connected subgraph $B$ of $G$ such that $B$ does not have a cut-vertex.
%A \emph{block} of a graph $G$ is a maximal subgraph $B$ of $G$ such that $B$ is $2$-connected or isomorphic to $K_2$.
%A \emph{block} of a graph $G$ is a maximal connected subgraph of $G$ for which every two edges are contained in a cycle.

\begin{restatable}{theorem}{brooksgen}
  \label{thm:brooks3gen}
  Let $G$ be a graph with maximal local edge-connectivity~$3$.
  Then $G$ is $3$-colourable if and only if each block of $G$
  cannot be obtained from an odd wheel by performing a (possibly empty) sequence of Haj\'os joins with an odd wheel.
\end{restatable}

%\noindent
For convenience, we call a graph that can be obtained from an odd wheel by performing a sequence of Haj\'os joins with odd wheels a {\it \fatras}.
Suppose that $G_1$ and $G_2$ are \fatras es.  It can be shown,
by a routine induction argument, 
that the Haj\'os join of
$G_1$ and $G_2$ is itself a \fatras.

%%In this paper we investigate a new parameterization for vertex colouring.
%Let \textsc{$k$-colourability} be the problem that takes an instance of
%a graph $G$, and answers ``yes'' or ``no'' according to the
%existence of a proper vertex $k$-colouring for $G$.
%\textsc{$3$-colourability} is NP-complete even when restricted to planar graphs with vertices of degree at most $4$~\cite{Garey1976}. Moreover,
%for any fixed $k\geq 3$,
%\textsc{$k$-colourability} is NP-complete. %,
%%so unless P=NP, $k$-colouring cannot be parameterized by $k$.
%%However, parameterizations of this problem are known to be ;
%%But for other parameters, it can be FPT.
%However, some parameterized versions of \textsc{$k$-colourability} have been shown to be fixed-parameter tractable (FPT); that is, for a graph of size $n$ and a parameter $p$, the problem is decidable in $f(p)n^c$ time for some constant $c$.
%For instance, \textsc{$k$-colourability} is FPT when parameterized by the treewidth of the graph.
%In this paper we consider a new parameterization of \textsc{$k$-colourability}, and show that the corresponding parameterized problem is FPT.

It follows from \cref{main,thm:brooks3gen} that there is a polynomial-time algorithm that finds a $k$-colouring for a $k$-connected graph with maximal local edge-connectivity~$k$, or determines that no such colouring exists; and there is a polynomial-time algorithm for finding an optimal colouring of any graph with maximal local edge-connectivity~$3$.

A graph in ${\mathcal C}_2^k$ is also said to have \emph{maximal local connectivity~$k$}.
%A graph $G$ has \emph{maximal local connectivity $k$} when $\lc(x,y) \leq k$ for all distinct $x, y \in V(G)$.
These graphs have been studied previously; primarily, the problem of determining bounds on the maximum number of possible edges in a graph with $n$ vertices and maximal local connectivity~$k$ has received much attention (see~\cite{Bollobas2004,Leonard1973,Mader1973a,Sorensen1974}).
Note that for a $k$-connected graph $G$ with maximal local connectivity~$k$ (that is, for $G$ in $\widehat{\mathcal{C}}_2^k$), we have $\lc(x,y) = k$ for all distinct $x, y \in V(G)$.
When $k=3$, it turns out that %the class of $3$-connected graphs with maximal local edge-connectivity~$3$ coincides with the class of $3$-connected graphs with maximal local connectivity~$3$
$\widehat{\mathcal{C}}_1^3 = \widehat{\mathcal{C}}_2^3$
(see \cref{structurelemma1}).
This leads to the following:
% theorem.

\begin{theorem}
  \label{colour_nl4c}
  Let $G$ be a $3$-connected graph with maximal local connectivity~$3$.  Then $G$ is $3$-colourable if and only if $G$ is not an odd wheel.  Moreover, there is a polynomial-time algorithm that finds an optimal colouring for $G$.
\end{theorem}

\noindent
However, we give an example in \cref{secmlc3} to demonstrate that $\widehat{\mathcal{C}}_1^4 \neq \widehat{\mathcal{C}}_2^4$ (see \cref{fig:noedge5lc}).

%In fact,
The class $\widehat{\mathcal C}_3^k$ is well known.
A graph $G$ is \emph{minimally $k$-connected} if it is $k$-connected and the removal of any edge leads to a graph that is not $k$-connected.
It is easy to check that a graph is in $\widehat{\mathcal C}_3^k$ if and only if it is minimally $k$-connected (see, for example, \cite[Lemma 4.2]{Bollobas2004}).

We now review known results regarding the bounds on the chromatic number of these classes.
Mader proved that any graph with at least one edge contains a pair of adjacent vertices whose local connectivity is equal to the minimum of their degrees~\cite{Mader1973}.
It follows that any graph in ${\mathcal C}_3^k$ has a vertex of degree at most $k$.
This, in turn, implies that a graph %with maximal local connectivity $k$
in $\mathcal{C}_3^k$ %, or $\mathcal{C}_2^k$, 
is $(k+1)$-colourable.
In particular, minimally $k$-connected graphs, and graphs with maximal local connectivity~$k$, %and graphs with maximal local edge-connectivity~$k$
are all $(k+1)$-colourable.

%Another result of Mader~\cite{mader:4k} states that a graph with
%average degree at least $4k$ contains a $(k+1)$-connected subgraph.
%This can be rephrased as follows: graphs in ${\cal C}_4^k$ have average degree less than $4k$, and therefore chromatic number at most $4k$.
%The sharp bound for this result is not known; the best known bound is due to %Hajnal~\cite{hajnal:ext}.
%Bernshteyn and Kostochka~\cite{Bernshteyn}.

Despite these results, it seems that, so far, the tractability of %the problem of
%deciding if a graph is $k$-colourable (\textsc{$k$-colourability}), or
%finding a $k$-colouring or optimal colouring,
computing the chromatic number, %in polynomial time
or finding a $k$-colouring,
for a graph in one of these classes
has not been investigated.
%As we see now, this problem is often intractable, but there are some exceptions. 
%such as \cref{colour_nl4c}.
For fixed~$k$, let \textsc{$k$-colouring} be the search problem that, given a graph $G$, finds a $k$-colouring for $G$, or determines that none exists.
%Let \textsc{$k$-colourability} be the problem that takes an instance of
%a graph $G$, and answers ``yes'' or ``no'' according to the
%existence of a proper vertex $k$-colouring for $G$.
An overview of our findings in this paper is given in \cref{overviewfig}, where we illustrate the complexity of \textsc{$k$-colouring} when restricted to the various classes defined by constraints on connectivity.

%%%
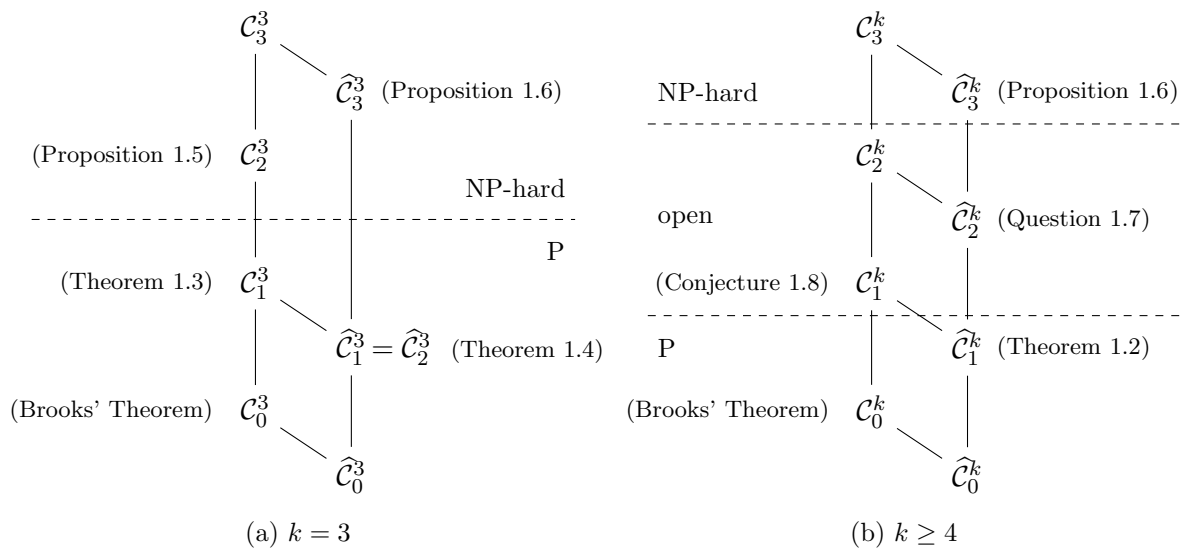
\begin{figure}
  \begin{subfigure}{0.48\textwidth}
    \centering
    \begin{tikzpicture}[scale=0.85]
      %\node (c3k) at (0,4) {$\mathcal{C}_4^3$};
      %\node (ch3k) at (1.5,3) {$\widehat{\mathcal{C}}_4^3$};
      \node (c2k) at (0,2) {$\mathcal{C}_3^3$};
      \node (ch2k) at (1.5,1) {$\widehat{\mathcal{C}}_3^3$};
      \node[anchor=west] (ch2klabel) at (1.5,1) {\ \ {\footnotesize(\cref{m3c_npc})}};
      \node (c1k) at (0,0) {$\mathcal{C}_2^3$};
      \node[anchor=east] at (-0.5,0) {\footnotesize{(\cref{colour_n4c_2})}};
      \node (c0k) at (0,-2) {$\mathcal{C}_1^3$};
      \node[anchor=east] at (-0.5,-2) {\footnotesize{(\cref{thm:brooks3gen})}};
      \node (ch0k) at (1.5,-3) {$\widehat{\mathcal{C}}_1^3$};
      \node[anchor=west] (ch0klabel) at (1.5,-3) {$\ = \widehat{\mathcal{C}}_2^3$ \ {\footnotesize(\cref{colour_nl4c})}};
      \node (cn1k) at (0,-4) {$\mathcal{C}_{0}^3$};
      \node[anchor=east] at (-0.5,-4) {\footnotesize{(Brooks' Theorem)}};
      \node (chn1k) at (1.5,-5) {$\widehat{\mathcal{C}}_{0}^3$};
      \draw (c2k) -- (c1k) -- (c0k) -- (cn1k) -- (chn1k)
      (ch2k) -- (ch0k) -- (chn1k)
      (c2k) -- (ch2k)
      (c0k) -- (ch0k);
      \draw[dashed] (-3.5, -1) -- (5, -1);
      \node[anchor=east] at (5, -0.5) {\small NP-hard};
      \node[anchor=east] at (5, -1.5) {\small P};
    \end{tikzpicture}
    \subcaption{$k=3$}
  \end{subfigure}
  \begin{subfigure}{0.48\textwidth}
    \centering
    \begin{tikzpicture}[scale=0.85]
      %\node (c3k) at (0,4) {$\mathcal{C}_4^k$};
      %\node (ch3k) at (1.5,3) {$\widehat{\mathcal{C}}_4^k$};
      \node (c2k) at (0,2) {$\mathcal{C}_3^k$};
      \node (ch2k) at (1.5,1) {$\widehat{\mathcal{C}}_3^k$};
      \node[anchor=west] (ch2klabel) at (1.5,1) {\ \ {\footnotesize(\cref{m3c_npc})}};
      \node (c1k) at (0,0) {$\mathcal{C}_2^k$};
      \node (ch1k) at (1.5,-1) {$\widehat{\mathcal{C}}_2^k$};
      \node[anchor=west] (ch1klabel) at (1.5,-1) {\ \ {\footnotesize(\cref{q8})}};
      \node (c0k) at (0,-2) {$\mathcal{C}_1^k$};
      \node[anchor=east] at (-0.5,-2) {\footnotesize{(\cref{q9})}};
      \node (ch0k) at (1.5,-3) {$\widehat{\mathcal{C}}_1^k$};
      \node[anchor=west] (ch0klabel) at (1.5,-3) {\ \ {\footnotesize(\cref{main})}};
      \node (cn1k) at (0,-4) {$\mathcal{C}_{0}^k$};
      \node[anchor=east] at (-0.5,-4) {\footnotesize{(Brooks' Theorem)}};
      \node (chn1k) at (1.5,-5) {$\widehat{\mathcal{C}}_{0}^k$};
      \draw (c2k) -- (c1k) -- (c0k) -- (cn1k) -- (chn1k)
      (ch2k) -- (ch1k) -- (ch0k) -- (chn1k)
      (c2k) -- (ch2k)
      (c1k) -- (ch1k)
      (c0k) -- (ch0k);
      \node[anchor=west] at (-3.5, 1) {\small NP-hard};
      \draw[dashed] (-3.5, 0.5) -- (5, 0.5);
      \node[anchor=west] at (-3.5, -1.0) {\small open};
      \draw[dashed] (-3.5, -2.5) -- (5, -2.5);
      \node[anchor=west] at (-3.5, -3) {\small P};
    \end{tikzpicture}
    \subcaption{$k \geq 4$}
  \end{subfigure}
  \caption{\textsc{$k$-colouring} complexity for graph classes defined by constraints on connectivity.}
  \label{overviewfig}
\end{figure}
%%%

If $k=1$, then ${\mathcal C}_3^k$ is the class of forests, so all the
classes are trivial.  For $k=2$, since it is easy to determine
if a graph is $2$-colourable, and %since as explained above
all graphs in ${\mathcal C}_3^k$ are 3-colourable, we may compute the
chromatic number of any graph in ${\mathcal C}_3^k$ in polynomial time.
%So the %, for $k=2$, the
%only open question is answered by the next \lcnamecref{npcc3} (which we prove in \cref{secnpc3}).

%\begin{proposition}
  %\label{npcc3}
  %It is NP-complete to decide whether a $2$-connected graph with no
  %$3$-connected subgraph is $3$-colourable.
%\end{proposition} 
%%\noindent In other words, \textsc{3-colourability} is NP-complete even when restricted to $2$-connected graphs in ${\cal C}_4^2$.

When $k=3$, \cref{colour_nl4c} implies that
\textsc{3-colouring} is polynomial-time solvable when restricted to $\widehat{\mathcal C}_2^3$.
For the class $\mathcal{C}_1^3$, this problem remains polynomial-time solvable, % (for graphs with any connectivity).
by \cref{thm:brooks3gen}.
  One might hope to generalise these results in one of two other possible directions: to the
class ${\mathcal C}_2^3$, or to
$\widehat{\mathcal C}_3^3$.  But any such attempt %will fail, unless $\textrm{P}=\textrm{NP}$, because of
is likely to fail, due to
the following results (see \cref{secmlc3,secnpc2} respectively):

\begin{proposition}
  \label{colour_n4c_2}
  For fixed $k\geq 3$, the problem of deciding if a
  $(k-1)$-connected graph with maximal local connectivity~$k$ is
  $3$-colourable is NP-complete.
\end{proposition}

\begin{proposition}
  \label{m3c_npc}
  For fixed $k\geq 3$, the problem of deciding if
  a minimally $k$-connected graph is $k$-colourable is NP-complete.
\end{proposition}

%\medskip

Now consider when $k \geq 4$.
It follows from \cref{main} that
%\textsc{$k$-colourability}, and in fact
\textsc{$k$-colouring} is polynomial-time solvable when restricted to $\widehat{\mathcal C}_1^k$. % (see \Cref{algo}).
However, the complexity for the more general class $\widehat{\mathcal C}_2^k$ remains an interesting open problem:
%Although \cref{colour_n4c_2} rules out,
%under the assumption that $\textrm{P}\neq \textrm{NP}$,
%the possibility of 
%a polynomial-time algorithm for \textsc{$3$-colourability} when restricted to the $(k-1)$-connected members of ${\cal C}_2^k$,
%the complexity of \textsc{$3$-colourability}, or \textsc{$k$-colourability}, when restricted to $k$-connected graphs in ${\cal C}_2^k$ for $k\geq 4$
%remains an open problem.
\begin{question}
  \label{q8}
  For fixed $k\geq 4$, %the problem of
  is there a polynomial-time algorithm that, given a $k$-connected
  graph $G$ with maximal local connectivity~$k$,
  finds a $k$-colouring of $G$, or determines that none exists?
\end{question}

\noindent
We also show that \textsc{$3$-colourability} is NP-complete for a graph in $\mathcal{C}_1^k$, when $k\geq 4$, so computing the chromatic number for a graph in this class, or in $\mathcal{C}_2^k$, is NP-hard, as is \textsc{$3$-colouring}.  However, the complexity of \textsc{$k$-colouring} (or \textsc{$k$-colourability}) for these classes is unresolved.  We make the following conjecture:
%We also show that it is NP-hard to compute the chromatic number for a graph in $\mathcal{C}_1^k$, when $k\geq 4$.
%However the complexity of \textsc{$k$-colouring} for this class is unresolved.
%%However the complexity of \textsc{$k$-colourability} when restricted to the entire class $\mathcal{C}_1^k$ is also unresolved for $k\geq 4$.

\begin{conjecture}
  \label{q9}
  For fixed $k\geq 4$, %the problem of
  there is a polynomial-time algorithm that, given a 
  graph $G$ with maximal local edge-connectivity~$k$,
  finds a $k$-colouring of $G$, or determines that none exists.
\end{conjecture}

\noindent
Stiebitz and Toft have recently announced a resolution to this conjecture in the affirmative~\cite{Stiebitz2016}.

It is worth noting that the class $\widehat{\mathcal C}_1^k$ is non-trivial.
%When $k=3$, all $3$-connected cubic graphs are members of the class, as are $3$-connected graphs with $n-1$ vertices of degree~$3$ and a single vertex of degree more than $3$.
All $k$-connected $k$-regular graphs are members of the class, as are $k$-connected graphs with $n-1$ vertices of degree~$k$ and a single vertex of degree more than $k$.
%A member of ${\cal C}_1^3$ can have chromatic number between $2$ and $4$ (for instance, odd wheels). 
A member of the class
can have arbitrarily many vertices of degree at least~$k+1$.
To see this for $k=3$, consider a graph $G_{3,x}'$, for $x \geq 3$, that is obtained from a grid graph $G_{3,x}$ (the Cartesian product of path graphs on $3$ and $x$ vertices) by adding two vertex-disjoint edges linking vertices of degree~$2$ at distance~$2$.  The graph $G_{3,x}'$ is in $\widehat{\mathcal C}_1^3$, and has $x-2$ vertices of degree~$4$.
A similar example can be constructed for any $k>3$; for example, see \cref{figk4} for when $k=4$. 
\begin{figure}
  \centering
  \begin{tikzpicture}[scale=1]
    \tikzset{VertexStyle/.append style = {draw,minimum height=7,minimum width=7,fill=black}}
    \SetVertexNoLabel
    \Vertex[x=-1,y=1.5]{z}
    \Vertex[x=0,y=0]{a1}
    \Vertex[x=0,y=1]{a2}
    \Vertex[x=0,y=2]{a3}
    \Vertex[x=0,y=3]{a4}
    \Vertex[x=2,y=0]{b1}
    \Vertex[x=2,y=1]{b2}
    \Vertex[x=2,y=2]{b3}
    \Vertex[x=2,y=3]{b4}
    \Vertex[x=2,y=4]{b5}
    \Vertex[x=4,y=0]{c1}
    \Vertex[x=4,y=1]{c2}
    \Vertex[x=4,y=2]{c3}
    \Vertex[x=4,y=3]{c4}
    \Vertex[x=4,y=4]{c5}
    \Vertex[x=6,y=0]{d1}
    \Vertex[x=6,y=1]{d2}
    \Vertex[x=6,y=2]{d3}
    \Vertex[x=6,y=3]{d4}
    \Vertex[x=7,y=1.5]{z2}

    \Edge(z)(a1)
    \Edge(z)(a2)
    \Edge(z)(a3)
    \Edge(z)(a4)
    \Edge(a1)(a2)
    \Edge(a2)(a3)
    \Edge(a3)(a4)
    \Edge(a1)(b1)
    \Edge(a2)(b2)
    \Edge(a3)(b3)
    \Edge(a4)(b4)
    \Edge(b1)(b2)
    \Edge(b2)(b3)
    \Edge(b3)(b4)
    \Edge(b4)(b5)
    \Edge(b2)(c2)
    \Edge(b3)(c3)
    \Edge(b4)(c4)
    \Edge(b5)(c5)
    \Edge(c1)(c2)
    \Edge(c2)(c3)
    \Edge(c3)(c4)
    \Edge(c4)(c5)
    \Edge(z2)(d1)
    \Edge(z2)(d2)
    \Edge(z2)(d3)
    \Edge(z2)(d4)
    \Edge(d1)(d2)
    \Edge(d2)(d3)
    \Edge(d3)(d4)
    \tikzset{EdgeStyle/.append style = {dashed}}
    \Edge(c1)(d1)
    \Edge(c2)(d2)
    \Edge(c3)(d3)
    \Edge(c4)(d4)
    \tikzset{EdgeStyle/.append style = {solid,bend left}}
    \Edge(a1)(a4)
    \Edge(b1)(b5)
    \Edge(c1)(c5)
    \Edge(d1)(d4)
    \tikzset{EdgeStyle/.append style = {bend right}}
    \Edge(b1)(b3)
    \Edge(b3)(b5)
    \Edge(c1)(c3)
    \Edge(c3)(c5)
  \end{tikzpicture}
  \caption{A $4$-connected graph with maximal local edge-connectivity~$4$, and arbitrarily many vertices of degree more than $4$.}
  \label{figk4}
\end{figure}
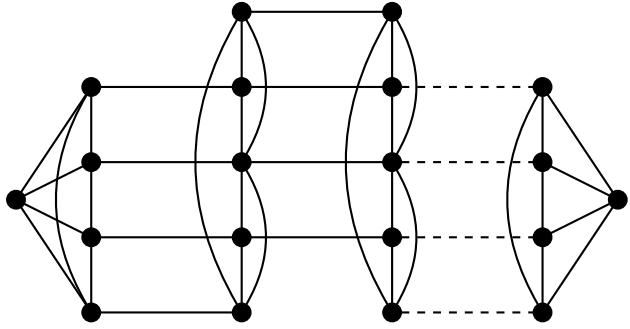

\medskip

Finally, we consider a parameterization of \textsc{$k$-colouring} based on the number $p_k$
  of vertices of degree at least $k+1$.
  By Brooks' theorem, a graph $G$ for which $p_k(G) = 0$ can be $k$-coloured in
polynomial time, unless it is a complete graph or an odd cycle.
We extend this to larger values of $p_k$, showing that,
  even when $k$ is part of the input,
  finding a $k$-colouring for a graph is fixed-parameter tractable (FPT) when parameterized by $p_k$.
  %the corresponding parameterized problem is fixed-parameter tractable (FPT).
  %We prove the following:

%However, some parameterized versions of \textsc{$k$-colourability} have been shown to be fixed-parameter tractable (FPT); that is, for a graph of size $n$ and a parameter $p$, the problem is decidable in $f(p)n^c$ time for some constant $c$.
%For instance, \textsc{$k$-colourability} is FPT when parameterized by the treewidth of the graph.
%In this paper we consider a new parameterization of \textsc{$k$-colourability}, and show that the corresponding parameterized problem is FPT.

%The parameter $p_k(G)$ we propose here is the number of
%vertices of degree greater than $k$ in $G$.  Trivially, by the greedy
%algorithm, a graph such that $p_k(G) = 0$ can be $(k+1)$-coloured in
%linear time, and by a classical theorem of Brooks, it can be
%$k$-coloured in polynomial time, unless it is a clique or an odd
%cycle.
%In this paper, we extend this to larger values of $p_k$,
%proving the following theorem:
%We propose the following extension to larger values of $p_k$,

\begin{theorem}
  \label{fewbigvs}
  Let $G$ be a graph with at most $p$ vertices of degree more than
  $k$.  There is a $\min\{k^p,p^p\} \cdot O(n+m)$-time algorithm for $k$-colouring $G$,
  or determining no such colouring exists.
\end{theorem}

%\medskip

This paper is structured as follows.
In the next section, we give preliminary definitions.
In \cref{secmlec}, we consider graphs with maximal local edge-connectivity~$k$, and prove \cref{main,thm:brooks3gen}.
We then consider the more general class of graphs with maximal local connectivity $k$, in \cref{secmlc3}, and prove \cref{colour_nl4c,colour_n4c_2}.
%In \cref{secmlc3}, we present the algorithm for $k$-colouring $k$-connected graphs with maximal local edge-connectivity~$k$, and prove \cref{colour_nl4c,colour_mlek}.
%In the following two sections, we present %NP-completeness or NP-hardness results; namely,
We present
the proof of \cref{m3c_npc} in \cref{secnpc2}.
Finally, in \cref{secbigvs}, we consider the problem of $k$-colouring a graph parameterized by the number of vertices of degree at least $k+1$, and prove \cref{fewbigvs}.

\section{Preliminaries}
%%%%%%%%%%%%%

Our terminology and notation follows \cite{bondyAndMurty} unless otherwise specified.
Throughout, we assume all graphs are simple.  
%Let $G$ be a graph. For $x,y \in V(G)$, an \emph{$xy$-path} is a path between $x$ and $y$.
%
We say that paths are \emph{internally disjoint} if they have no internal vertices in common.
A \emph{$k$-edge cut} is a $k$-element set $S \subseteq E(G)$ for
which $G \ba S$ is disconnected.  A \emph{$k$-vertex cut} is a
$k$-element subset $Z \subseteq V(G)$ for which $G - Z$ is
disconnected.
We call the vertex of a $1$-vertex cut a \emph{cut-vertex}.
For distinct non-adjacent vertices $x$ and $y$, and
$Z \subseteq V(G) \setminus \{x,y\}$, we say that \emph{$Z$ separates
  $x$ and $y$} when $x$ and $y$ belong to different components of
$G - Z$.  More generally, for disjoint, non-empty $X,Y,Z \subseteq V(G)$, we say
that \emph{$Z$ separates $X$ and $Y$} if, for each $x \in X$ and
$y \in Y$, the vertices $x$ and $y$ are in different components of
$G - Z$.  We call a partition $(X, Z, Y)$ of $V(G)$ a {\it $k$-separation}
if $|Z| \leq k$ and $Z$ separates $X$ from $Y$.
When $G$ is $k$-connected and $(X, Z, Y)$ is a $k$-separation of $G$, we have that $|Z| = k$.
By Menger's theorem, if $\kappa(x,y)=k$ for non-adjacent vertices $x$ and $y$, then there is a $k$-vertex cut that separates $x$ and $y$.
If $\kappa(x,y)=k\geq 2$ for adjacent vertices $x$ and $y$, then there is a $(k-1)$-vertex cut in $G \ba xy$ that separates $x$ and $y$.
We use these
freely %, in particular
in the proof of \cref{structurelemma1}.

We view a proper $k$-colouring of a graph $G$ as a function
$\phi: V(G) \rightarrow \{1,2,\dotsc,k\}$ where for every $uv \in E(G)$ we
have $\phi(u) \neq \phi(v)$.
For $X \subseteq V(G)$, we write $\phi(X)$ to denote the image of $X$ under $\phi$.

Given graphs $G_1$ and $G_2$, the graph with vertex set $V(G_1) \cup V(G_2)$ and edge set $E(G_1) \cup E(G_2)$ is denoted $G_1 \cup G_2$.

A \emph{diamond} is a graph obtained by removing an edge from $K_4$.
We call the two degree-$2$ vertices of a diamond $D$ the \emph{pick} vertices of $D$.

%\begin{lemma}
  %\label{2c_ld}
  %Let $G$ be a $2$-connected graph with at most one vertex of degree more than $k$, and at least one vertex of degree less than $k$.  Then $G$ is $k$-colourable.
%\end{lemma}

\section{Graphs with maximal local edge-connectivity~$k$}
\label{secmlec}
%%%%%%%%%%%%%%%%%%%%%%%%%%%%

In this section we prove \cref{main,thm:brooks3gen}.

Lov\'asz provided a short proof of Brooks' theorem in \cite{Lovasz1975}.
The proof can easily be adapted to show that graphs with at most one vertex of degree more than $k$ are often $k$-colourable.  We make this precise in the next lemma; the proof is provided for completeness.  A vertex is \emph{dominating} if it is adjacent to every other vertex of the graph.

\begin{lemma}
  \label{3c_nd}
  Let $G$ be a $3$-connected graph with at most one vertex of degree more than $k$, for $k \geq 3$, and no dominating vertices.  Then $G$ is $k$-colourable.
\end{lemma}
\begin{proof}
  Let $h$ be a vertex of $G$ with maximum degree.  Since $G$ has no dominating vertices and is connected, there is a vertex $y$ at distance two from $h$.  Let $z_1$ be a common neighbour of $h$ and $y$.
  Since $G$ is $3$-connected, $G-\{h,y\}$ is connected.  Let $z_1,z_2,\dotsc,z_{n-2}$ be a search ordering of $G-\{h,y\}$ starting at $z_1$; that is, an ordering of $V(G-\{h,y\})$ where each vertex $z_i$, for $2 \leq i \leq n-2$, has a neighbour $z_j$ with $j < i$.
  We colour $G$ as follows.  Assign $h$ and $y$ the colour $1$, say.  We can then (greedily) assign one of the $k$ colours to each of $z_{n-2}, z_{n-3}, \dotsc, z_2$ in turn, since at the time one of these vertices is considered, it has at most $k-1$ neighbours that have already been assigned colours.
  Finally, we can colour $z_1$, since it has degree at most $k$, but at least two of its neighbours, $h$ and $y$, are the same colour.
\end{proof}

Now we show that we can decompose a $k$-connected graph with maximal local edge-connectivity~$k$ into components each containing a single vertex of degree more than $k$.

\begin{lemma}
  \label{structurelemma2}
  Let $G$ be a $k$-connected graph with maximal local edge-connectivity~$k$, for $k \geq 3$, and at least two %non-adjacent
  vertices of degree more than $k$.
  Then there exists a $k$-edge cut $S$ such that one component of $G \ba S$
  %consists of at most $n/2$ vertices, and
  contains precisely one vertex of degree more than $k$,
  and the edges of $S$ are vertex disjoint.
  %Moreover, such a $k$-edge cut can be found in $O(nm)$ time.
\end{lemma}

\begin{proof}
  We say that a set of vertices $X_1 \subseteq V(G)$ is \emph{good} if $|X_1| \le n/2$ and $d(X_1)=k$, where $d(X_1)$ is the number of edges with one end in $X_1$ and the other end in $V(G) \setminus X_1$.
  If two good sets~$X_1$ and~$X_2$ have non-empty intersection, then $|X_1 \cup X_2| < n$, so $d(X_1 \cup X_2) \geq k$ by $k$-connectivity.
  As $d(X_1) + d(X_2) \ge d(X_1 \cup X_2) + d(X_1 \cap X_2)$ (see, for example, \cite[Exercise~2.5.4(b)]{bondyAndMurty}), it follows that $d(X_1 \cap X_2) = k$.
  %By submodularity, $d(X_1) + d(X_2) \ge d(X_1 \cup X_2) + d(X_1 \cap X_2)$ (see, for example, \cite[Exercise~2.5.4(b)]{bondyAndMurty}), and it follows that $d(X_1 \cap X_2) = k$.
  %and it follows, by submodularity (see, for example, \cite[Exercise~2.5.4(b)]{bondyAndMurty}), that $d(X_1 \cap X_2) = k$.
  Thus, if a good set~$X_1$ meets a good set~$X_2$, then $X_1 \cap X_2$ is also good.
  This implies that if a vertex of degree more than $k$ is in a good set, then there is unique minimal good set containing it.
  Since there is a $k$-edge cut between any two vertices,
  one of any two vertices is in a good set. Thus,
  all but at most one vertex of $G$ is in a good set.
  Let $X$ be a minimal good set containing at least one vertex of degree more than $k$.
  Suppose $X$ contains distinct vertices $x$ and $y$, each with degree more than $k$.
  Then there is $k$-edge cut separating them, so there is a good set containing exactly one of them.  By taking the intersection of this good set with $X$, we obtain a good set that is a proper subset of $X$ and contains at least one vertex of degree more than $k$; a contradiction.
  So $X$ contains precisely one vertex of degree more than $k$.
  Now $d(X)=k$, since $X$ is good, hence the $k$ edges with one end in $X$ and the other in $E(G)-X$ give an edge cut $S$.

  It remains to show that the edges of $S$ are vertex disjoint.
  Set $Y=V(G)\setminus X$, and let $X_S$ (respectively, $Y_S$) be the set of vertices of $X$ (respectively, $Y$) incident to an edge of $S$.
  %If $|X| \leq k$, then there are at most $k(k-1)/2 + k$ edges that are incident to at least one vertex of $X$.  But $k(k-1)/2 + k = k(k+1)/2 \leq k^2$; a contradiction, since each vertex of $X$ has degree at least $k$, and some vertex has degree more than $k$.
  Let $|X|=q$.
  Since every vertex in $X$ has degree at least $k$, and $X$ contains some vertex of degree more than $k$, we have that $\Sigma_{v\in X} d(v) \geq qk + 1$.
  If $q \leq k$, then, since each vertex in $X$ has at most $q-1$ neighbours in $X$, we have that $\Sigma_{v\in X} d(v) \leq q(q-1) + k \leq k(q-1)+k = qk$; a contradiction.
  So $X_S\neq X$ and, similarly, $Y_S\neq Y$.
  %
  %Observe that $X_S\neq X$, since every vertex of $X$ has degree at least $k$, and some vertex has degree more than $k$. Similarly, $Y_S\neq Y$.
Now, since $G$ is $k$-connected, there are $k$ internally disjoint paths from any vertex in $X \setminus X_S$ to any vertex in $Y\setminus Y_S$. 
Each of these paths must contain a different edge of $S$.
Thus $S$ satisfies the requirements of the lemma.
%
  %Given two vertices $x$ and $y$ of degree more than $k$, we can find a $k$-edge cut that separates $x$ and $y$ in $O(km)$ time,
  %by an application of the Ford-Fulkerson algorithm. % with $x$ as the source and $y$ as the sink.
  %Without loss of generality, $x$ is contained in a good set $X$.
  %%If $X$ contains a vertex $x'$ of degree more than $k$, distinct from $x$,
  %%we can, again, apply the Ford-Fulkerson algorithm, this time with $x$ as the source and $x'$ as the sink, and obtain, by taking the intersection, a set $X'$
  %By $O(n)$ applications of the Ford-Fulkerson algorithm, we can obtain a minimal good set containing $x$.  Thus we can find the required $k$-edge cut in $O(knm)$ time.
\end{proof}

Next we show, loosely speaking, that if a graph $G$ has a 
$k$-edge cut $S$ where the edges in $S$ have no vertices in common, then
the problem of
 $k$-colouring $G$ can essentially be reduced to finding $k$-colourings of the components of $G \ba S$; %that satisfy an easily enforceable constraint.
 the only bad case
 is when the vertices incident to $S$ are coloured all the same colour in one component, and all different colours in the other.

\begin{lemma} \label{lnew}
  %Let $k \ge 3$, and let $G$ be a connected graph, where $V_1=\{u_1,u_2,\dotsc,u_k\}$ and $V_2=\{v_1,v_2,\dotsc,v_k\}$ are $k$-element disjoint subsets of $V(G)$, and $S=\{u_1v_1, u_2v_2, \dotsc, u_kv_k\}$ is a $k$-edge cut such that $G \ba S$ consists of two components $G_u$ and $G_2$ with $U \subseteq V(G_u)$ and $V \subseteq V(G_2)$.
  %
  Let $G$ be a connected graph with a $k$-edge cut $S$, for $k \geq 3$, such that the edges of $S$ are vertex-disjoint, and $G \ba S$ consists of two components $G_1$ and $G_2$.  %Let $(V_1,V_2)$ be a partition of the vertices incident to an edge in $S$ such that $V_1 \subseteq V(G_1)$ and $V_2 \subseteq V(G_2)$.
  Let $V_i$ be the set of vertices in $V(G_i)$ incident to an edge of $S$, for $i \in \{1,2\}$.
  \begin{enumerate}[label=\rm (\roman*)]
    \item Then $G$ is $k$-colourable if and only if there exists a $k$-colouring $\phi_1$ of $G_1$ and a $k$-colouring $\phi_2$ of $G_2$ such that $\{|\phi_1(V_1)|,|\phi_2(V_2)|\} \neq \{1,k\}$.\label{part1}
    \item Moreover, if 
  $\phi_1$ and $\phi_2$ are $k$-colourings of $G_1$ and $G_2$, respectively, for which
  $\{|\phi_1(V_1)|,|\phi_2(V_2)|\} \neq \{1,k\}$, then there exists a permutation $\sigma$ such that
\begin{equation*}
  \phi(x) = 
  \begin{cases}
    \phi_1(x) & \textrm{for } x \in V(G_1), \\
    \sigma(\phi_2(x)) & \textrm{for } x \in V(G_2)
  \end{cases}
\end{equation*}
is a $k$-colouring of $G$.\label{part2}
  \end{enumerate}
\end{lemma}

\begin{proof}
  First, we prove \ref{part2}, which implies that \ref{part1} holds in one direction.
  Let $\phi_1$ and $\phi_2$
be $k$-colourings 
  of $G_1$ and $G_2$, respectively, 
  for which $\{|\phi_1(V_1)|,|\phi_2(V_2)|\} \neq \{1,k\}$.
  We will construct an auxiliary graph $H$ where the vertices are labelled by subsets of $V_1$ or $V_2$ in such a way that if we can $k$-colour $H$, then there exists a permutation $\sigma$ such that 
 $\phi$, as defined in the statement of the lemma, is a $k$-colouring of $G$.
 % A vertex in this auxiliary graph will represent the set of vertices in $V_1$ given the same colour in $G_1$, or the set of vertices in $V_2$ given the same colour in $G_2$.
%

  Let $(T_1,T_2,\dotsc,T_{|\phi_1(V_1)|})$ be the partition of the vertices in $V_1$ into colour classes with respect to $\phi_1$ and,
  likewise, let $(W_1,W_2,\dotsc,W_{|\phi_2(V_2)|})$ be the partition of $V_2$ into colour classes with respect to $\phi_2$.
  %In the auxiliary graph $H$, we have a vertex
  We construct a graph $H$ consisting of $|\phi_1(V_1)| + |\phi_2(V_2)|$ vertices:
  for each $i \in \{1,2,\dotsc,|\phi_1(V_1)|\}$, we have a vertex $t_i \in V(H)$ labelled by $T_i$, and,
  for each $i \in \{1,2,\dotsc,|\phi_2(V_2)|\}$, we have a vertex $w_i \in V(H)$ labelled by $W_i$.
  Let $T = \{t_i : 1 \leq i \leq |\phi_1(V_1)|\}$ and let $W=\{w_i : 1 \leq i \leq |\phi_2(V_2)|\}$.
  Each $t \in T$ (respectively, $w \in W$) is adjacent to every vertex in $T-\{t\}$ (respectively, $W-\{w\})$.
  Finally, for each edge $v_1v_2$ in $S$, we add an edge between the vertex $t \in T$ labelled by the colour class containing $v_1$, and the vertex $w \in W$ labelled by the colour class containing $v_2$, omitting parallel edges.
  Thus there are at most $k$ edges between vertices in $T$ and vertices in $W$.

  Now we show that $H$ is $k$-colourable.
Consider a vertex $t \in T$.  If it has $x$ neighbours in $W$, then it represents a colour class consisting of at least $x$ vertices of $V_1$.  So there are at most $k-x$ vertices in $T-\{t\}$, and hence $t$ has degree at most $x + (k-x)$.
  It follows, by Brooks' theorem, that $H$ is $k$-colourable unless it is a complete graph, % or an odd cycle.
  as $k \ge 3$.
  %If $H$ is an odd cycle, then, since every vertex in $T$ has a neighbour in $W$, and vice versa, $H$ is a $3$-cycle, and hence is $k$-colourable for any $k \ge 3$.
  %So we may assume that $H$ is a complete graph.
  Moreover, 
  if $|V(H)| \leq k$, then $H$ is $k$-colourable, so assume that $|V(H)| > k$.
Then, without loss of generality, we may assume that $|T| > k/2$.  Since there are at most $k$ edges between vertices in $T$ and vertices in $W$, and each vertex of $T$ has the same number of neighbours in $W$, it follows that each vertex in $T$ has a single neighbour in $W$.  Since $H$ is a complete graph, we have $|W|=1$, and hence, recalling that $|V(H)| > k$, we have $|T|=k$.  That is, $|\phi_1(V_1)|=k$ and $|\phi_2(V_2)|=1$; a contradiction.

  Now $H$ is $k$-colourable.
  By permuting the colours of a $k$-colouring of $H$, we can obtain a $k$-colouring $\psi$ such that $\psi|_{V_1} = \phi_1$.  Then $\psi|_{V_2}$ induces a permutation $\sigma$ of $\phi_2$, in the obvious way, with the desired properties.
  This completes the proof of \ref{part2}.

  Finally, we observe that when $\{|\phi_1(V_1)|,|\phi_2(V_2)|\} = \{1,k\}$ for every $k$-colouring $\phi_1$ of $G_1$ and  every $k$-colouring $\phi_2$ of $G_2$, then $G$ is not $k$-colourable.  This completes the proof of \ref{part1}.
\end{proof}

Suppose that a graph $G$ has a $k$-edge cut $S$ that separates $X$ from $Y$, where $(X,Y)$ is a partition of $V(G)$.  We fix the following notation for the remainder of this section. Let $Y_S$ (respectively, $X_S$) be the subset of $Y$ (respectively, $X$) consisting of vertices incident to an edge in $S$.
  Let $G_{X}$ (respectively, $G_Y$) be the graph obtained from $G[X \cup Y_S]$ (respectively, $G[Y \cup X_S]$) by adding edges so that $Y_S$ (respectively, $X_S$) is a clique.

\begin{lemma}
  \label{close}
  Let $G$ be a $k$-connected graph, for $k \geq 3$, with maximal local edge-connectivity~$k$, and a $k$-edge cut $S$ that separates $X$ from $Y$, where $(X,Y)$ partitions $V(G)$. %Let $Y_S$ be the subset of $Y$ consisting of vertices incident to an edge in $S$, and let $G_X$ be the graph obtained from $G[X \cup Y_S]$ by adding edges so that $Y_S$ is a clique.
  Then $G_X$ is $k$-connected and has maximal local edge-connectivity~$k$.
\end{lemma}
\begin{proof}
  %Let $X_S$ be the subset of $X$ consisting of vertices incident to an edge in $S$.
  First we show that $G_X$ has maximal local edge-connectivity~$k$.
  The only vertices of degree more than $k$ in $G_X$ are in $X$.
  Suppose $u$ and $v$ are vertices in $X$ of degree more than $k$. 
  Clearly, for each $uv$-path in $G_X[X]$ there is a corresponding $uv$-path in $G[X]$.
  We show that there are at least as many edge-disjoint $uv$-paths that pass through an edge of $S$ in $G$ as there are in $G_X$; it follows that $\lambda_{G_X}(u,v) \leq \lambda_{G}(u,v) \leq k$.
  Since $S$ is a $k$-edge cut in $G_X$, there are at most $\lfloor{k/2\rfloor}$ edge-disjoint paths in $G_X$ starting and ending at a vertex in $X_S$.
  Let $y$ be a vertex in $Y$.
  Since $G$ is $k$-connected, the Fan Lemma (see, for example, \cite[Proposition~9.5]{bondyAndMurty}) implies that there are $k$ paths from $y$ to each member of $X_S$ that meet only in $y$.  Hence, there are $\lfloor{k/2\rfloor}$ edge-disjoint paths
 in $G[Y \cup X_S]$ starting and ending at a vertex in $X_S$.
 Thus, we deduce that $G_X$ has maximal local edge-connectivity~$k$.

  We now show that $G_X$ is $k$-connected, by demonstrating that $\lc_{G_X}(u,v) \geq k$ for all distinct $u,v \in V(G_X)$.
  First, suppose that $u,v \in X$.
  Evidently, for each $uv$-path in $G[X]$ there is a corresponding $uv$-path in $G_X[X]$.
  Moreover, each $uv$-path in $G$ that traverses an edge of $S$ traverses two such edges $xy$ and $x'y'$, say, where $x,x' \in X_S$ and $y,y' \in Y_S$.  By replacing the $x'y'$-path in $G$ with the edge $x'y'$ in $G_X$, we obtain a $uv$-path of $G_X$.
  %Consider a set of edge-disjoint paths in $G$ that start and end at vertices in $X_S$, and traverse an edge of $S$.
  %There are at most $\lfloor{k/2\rfloor}$ such paths.
  %Since there are $\lfloor{k/2\rfloor}$ such paths in $G_X$,
  We deduce that $\lc_{G_X}(u,v) \geq \lc_G(u,v) \geq k$ for any $u,v \in X$.
  Now suppose $u, v \in Y_S$.  Then there are $k-1$ internally disjoint $uv$-paths in $G_X[Y_S]$.  Pick $u',v' \in X_S$ such that $uu'$ and $vv'$ are in $S$.
  Since $G_X[X]$ is connected, there is at least one $u'v'$-path in $G_X[X]$, so there are $k$ internally disjoint $uv$-paths in $G_X$.
  Finally, let $u \in X$ and $v \in Y_S$.  Since $G$ is $k$-connected,
  the Fan Lemma implies that
  there are $k$ paths from $u$ to each vertex of $Y_S$ in $G$ that meet only in $u$.  Hence there are $k$ such paths in $G_X$.
  Since $Y_S$ is a clique in $G_X$, there are $k$ internally disjoint $uv$-paths in $G_X$.
  Thus $\lc_{G_X}(u,v) \geq k$ for all distinct $u,v \in V(G_X)$, as required.
\end{proof}

\begin{proposition}
  \label{p1}
  Let $G$ be a $k$-connected graph, for $k \geq 3$, with maximal local edge-connectivity~$k$ and at least two vertices of degree more than $k$.  Then $G$ is $k$-colourable.
\end{proposition}
\begin{proof}
  The proof is by induction on the number of vertices of degree more than $k$.
  First we show that the \lcnamecref{p1} holds when $G$ has precisely two vertices of degree more than $k$.
  Let $x$ and $y$ be distinct vertices of $G$ with degree more than $k$.
  By \cref{structurelemma2}, there is a $k$-edge cut $S$ that separates $X$ from $Y$, where $x \in X$, $y \in Y$, $(X,Y)$ is a partition of $V(G)$, and $X$ contains precisely one vertex of degree more than $k$.
  Consider the graph~$G_{X}$; 
  this graph is $3$-connected by \cref{close}, and has no dominating vertices by definition.
  Hence, by \cref{3c_nd}, $G_{X}$ is $k$-colourable.
  Moreover, in such a $k$-colouring, the vertices in $Y_S$ are given $k$ different colours, since they form a $k$-clique, and hence the vertices in $X_S$ are not all the same colour.
  So $G_{X}[X] = G[X]$ is $k$-colourable in such a way that the vertices in $X_S$ are not all the same colour.
  By symmetry, $G[Y]$ is $k$-colourable in such a way that the vertices in $Y_S$ are not all the same colour.
  It follows, by \cref{lnew}, that $G$ is $k$-colourable.

  Now let $G$ be a graph with $p$ vertices of degree more than $k$, for $p > 2$.
  We assume that a $k$-connected graph with maximal local edge-connectivity~$k$, and $p-1$ vertices of degree more than $k$ is $k$-colourable.
  By \cref{structurelemma2}, there is a $k$-edge cut $S$ that separates $X$ from $Y$, where $X$ contains precisely one vertex $x$ of degree more than $k$, and $(X,Y)$ is a partition of $V(G)$.
  %%Let $X_S = \{x_1,x_2,\dotsc,x_k\}$ be a subset of $X$ and let $Y_S = \{y_1,y_2,\dotsc,y_k\}$ be a subset of $Y$ such that $S = \{x_1y_1, x_2y_2, \dotsc, x_ky_k\}$.
  %Again, we let $X_S$ (respectively, $Y_S$) be the subset of $X$ (respectively, $Y$) consisting of vertices incident to an edge in $S$, and 
  %let $G_{X}$ (respectively, $G_Y$) be the graph obtained from $G[X \cup Y_S]$ (respectively, $G[Y \cup X_S]$) by adding edges so that $X_S$ is a clique.
%%(similarly, $G_{Y}$ is the graph obtained from $G[Y \cup X_S]$ by adding edges so that $X_S$ induces a clique)
  The graph $G_Y$ is $k$-connected and has maximal local edge-connectivity~$k$, by \cref{close}.
  Thus, by the induction assumption, $G_Y$ is $k$-colourable.  It follows that $G[Y]$ is $k$-colourable in such a way that the vertices in $Y_S$ are not all the same colour.
  The graph $G_X$ is $3$-connected, by \cref{close}, so is $k$-colourable, by \cref{3c_nd}. So $G[X]$ is $k$-colourable in such a way that the vertices in $X_S$ are not all the same colour.
  Thus, by \cref{lnew}, $G$ is $k$-colourable.  The \lcnamecref{p1} follows by induction.
\end{proof}

\begin{proof}[Proof of \cref{main}]
  Clearly if $G$ is a complete graph, then $G$ is $K_{k+1}$ and is not $k$-colourable.
  If $G$ is an odd wheel, then, since $G$ is not $4$-connected, we have $k=3$, and $G$ is not $3$-colourable.
  This proves one direction.
  Now suppose $G$ is not $k$-colourable and has $p$ vertices of degree more than $k$.
  Then $p < 2$, by \cref{p1}.
  If $p = 0$, then $G$ is a complete graph, by Brooks' theorem (an odd cycle is not $k$-connected for any $k\geq 3$).  If $p=1$, then $G$ has a dominating vertex $v$, by \cref{3c_nd}.
  Since $G-\{v\}$ is not $(k-1)$-colourable, and $G-\{v\}$ has maximum degree $k-1$, it follows, by Brooks' theorem, that $G-\{v\}$ is a complete graph or an odd cycle.  Thus $G$ is a complete graph or an odd wheel.
\end{proof}

\begin{corollary}
  \label{algo}
  Let $G$ be a $k$-connected graph with maximal local edge-connectivity~$k$.
  There is a polynomial-time algorithm that finds a $k$-colouring for $G$ when $G$ is $k$-colourable, or a $(k+1)$-colouring otherwise.
\end{corollary}
\begin{proof}
  %Since there is a straightforward algorithm for $2$-colouring a graph, and it is easy to $4$-colour an odd wheel or $k$-colour a complete graph on $k$ vertices, it suffices to show that when $G$ is $k$-colourable, there is a polynomial-time algorithm for $k$-colouring it.
%
  Suppose $G$ has at most one vertex of degree more than $k$.  If $G$ has no dominating vertices, then the proof of \cref{3c_nd} leads to an algorithm for $k$-colouring $G$.
  Otherwise, when $G$ has a dominating vertex $v$, the problem reduces to finding a $(k-1)$-colouring for $G-\{v\}$,
  where $G-\{v\}$ has maximum degree~$k-1$.
  In either case, we have a linear-time algorithm for colouring $G$.

  When $G$ has at least two vertices $x$ and $y$ of degree more than $k$, we use the approach taken in the proof of \cref{p1}.
  %Then there is a $k$-edge cut that separates $x$ and $y$, by \cref{structurelemma2}.
  We can find a $k$-edge cut $S$ that separates $x$ and $y$ in $O(km)$ time,
  by an application of the Ford-Fulkerson algorithm. % with $x$ as the source and $y$ as the sink.
  Without loss of generality, $x$ is contained in a component of $G \ba S$ with at most $n/2$ vertices.
  It follows, by the proof of \cref{structurelemma2}, that 
  with $O(n)$ applications of the Ford-Fulkerson algorithm we can obtain an edge cut $S'$ such that $x$ is the only vertex of degree more than $k$ in one component $X$ of $G \ba S'$.  %Thus we can find the $k$-edge cut $S'$ as described in \cref{structurelemma2} in $O(knm)$ time.
  Thus we can find the desired $k$-edge cut $S'$  in $O(knm) = O(nm)$ time.
  Let $Y = V(G) \setminus X$, and
  let $G_X$ and $G_Y$ be as defined just prior to \cref{close}.
  As $G_X$ is $3$-connected by \cref{close}, and has no dominating vertices by definition,
  we can find a $k$-colouring $\phi_X$ for $G_X$ in linear time by \cref{3c_nd}.
  To find a $k$-colouring $\phi_Y$ for $G_Y$, if one exists, we repeat this process recursively.
  %By a routine induction argument, when $G$ is $k$-colourable we can likewise find a $k$-colouring $\phi_Y$ for $G_Y$ in linear time.
  Then, by \cref{lnew}, we can extend $\phi_Y$ to a $k$-colouring of $G$ by finding a permutation for $\phi_X$, which can be done in constant time.
  When $G$ has $p$ vertices of degree more than $k$, %we repeat this process $p-1$ times.
  this process takes $O(pnm)$ time.
  Since $p \leq n$, the algorithm runs in $O(n^2m)$ time.
\end{proof}

\subsection*{An extension of Brooks' theorem when $k=3$}

We now work towards proving \cref{thm:brooks3gen}.
Recall that a \fatras\ is either an odd wheel, or a graph that can be obtained from odd wheels by applying the Haj\'os join.  We restate the theorem here in terms of wheel morasses:
\begin{theorem}
  \label{thm:brooks3gen2}
  Let $G$ be a graph with maximal local edge-connectivity~$3$.
  Then $G$ is $3$-colourable if and only if each block of $G$ is not a \fatras.
\end{theorem}

%Recall \cref{thm:brooks3gen}:
%\brooksgen*
%
%
%We call a graph that can be obtained from an odd wheel by performing a sequence of Haj\'os joins with odd wheels, a {\it \fatras}.
%Let us now establish some properties of these graphs.
Let us now establish some properties of \fatras es.
A graph $G$ is \emph{$k$-critical} if $\chi(G) = k$ and every proper subgraph $H$ of $G$ has $\chi(H) < k$.

\begin{proposition}\label{prop:fatras}
  Let $G$ be a \fatras.  Then
\begin{enumerate}[label=\rm (\roman*)]
\item $G$ is $4$-critical, and
\item for every two distinct vertices $x$ and $y$, we have $\lambda(x,y)\geq 3$.
\end{enumerate}
\end{proposition}
\begin{proof}
(i) It is well known that the Haj\'os join of two $k$-critical graphs is $k$-critical (see, for example, \cite[Exercise~14.2.9]{bondyAndMurty}).
Since the odd wheels are $4$-critical, we immediately get, by induction, that every \fatras\ is $4$-critical.

%\medskip

(ii) We prove this by induction on the number of Haj\'os joins.
The result can easily be checked for odd wheels.

Assume now that $G$ is the Haj\'os join of $G_1$ and $G_2$ with respect to $(u_1,v_1)$ and $(u_2,v_2)$.
Let $x$ and $y$ be two vertices in $G$.
If $x\in V(G_1)$ and $y\in V(G_1)$, then, by the induction hypothesis, there are three edge-disjoint $xy$-paths in $G_1$.
If one them contains $v_1u_1$, then replace it by the concatenation of $v_1v_2$ and a $v_2u_2$-path in $G_2\ba u_2v_2$ (such a path exists since $\lambda_{G_2}(u_2,v_2)\geq 3$ by the induction hypothesis). %Similarly, we can replace $u_1v_1$ by the concatenation of a $u_2v_2$-path in $G_2\ba u_2v_2$, and $v_2v_1$.
This results in three edge-disjoint $xy$-paths, so $\lambda_G(x,y)\geq 3$.
Likewise, if $x\in V(G_2)$ and $y\in V(G_2)$, then $\lambda_G(x,y)\geq 3$.

Assume now that $x\in V(G_1)$ and $y\in V(G_2)$.
Let us prove the following:

\begin{sublemma}\label{claima}
In $G_1\ba u_1v_1$, there are three edge-disjoint paths $P_1$, $P_2$ and $P_3$ such that $P_1$ and $P_2$ are $xu_1$-paths and $P_3$ is an $xv_1$-path.
\end{sublemma}

\begin{proof}
By the induction hypothesis, there are three edge-disjoint $xu_1$-paths $R_1, R_2, R_3$ in $G_1$.
If $v_1\in V(R_1)\cup V(R_2)\cup V(R_3)$, then we may assume, without loss of generality, that $v_1\in V(R_3)$ and $u_1v_1\notin E(R_1)\cup E(R_2)$.
Hence $R_1$, $R_2$ and the $xv_1$-subpath of $R_3$ are the desired paths.
Now we may assume that
$v_1\notin V(R_1) \cup V(R_2) \cup V(R_3)$.  Let $Q$ be a shortest path from $z_1 \in V(R_1)\cup V(R_2)\cup V(R_3)$ to $v_1$ in $G\ba u_1v_1$ (such a path exists by our connectivity assumption). %, and let $z_1$ be its initial vertex.
Without loss of generality, $z_1\in V(R_3)$.
Hence the desired paths are $R_1$, $R_2$ and the concatenation of the $xz_1$-subpath of $R_3$ and $Q$.
This proves \cref{claima}.
\end{proof}

By \cref{claima} and symmetry, there are three edge-disjoint paths $Q_1$, $Q_2$ and $Q_3$ in $G_2\ba u_2v_2$ such that $Q_1$ and $Q_2$ are $u_2y$-paths and $Q_3$ is a $v_2y$-path.
The paths obtained by concatenating $P_1$ and $Q_1$; $P_2$ and $Q_2$; and $P_3$, $v_1v_2$ and $Q_3$ are three edge-disjoint $xy$-paths in $G$, so $\lambda_G(x,y)\geq 3$.
\end{proof}

\begin{proof} [Proof of \cref{thm:brooks3gen}]
%  Suppose that some block of $G$ can be obtained from $K_4$ or an odd wheel by performing a sequence of series connections, each with either $K_4$ or an odd wheel.
 % Let $G_0, G_1, \dotsc, G_q$ be the sequence of graphs by which $G$ is obtained; that is, $G_0$ is $K_4$ or an odd wheel, $G_q = G$, and $G_{i}$ is the series connection of $G_{i-1}$ and $H_{i-1}$, where $H_i$ is $K_4$ or an odd wheel, for $i \in \{1, \dotsc, q\}$.
 % Pick $j \in \{0, 1, \dotsc, q-1\}$, so $H_j$ is $K_4$ or an odd wheel.
 % Observe that, for every $uv \in E(H_j)$, the graph $H_j \ba uv$ is $3$-colourable and the vertices $u$ and $v$ are given the same colour in any $3$-colouring.
  %It follows that $G_{j+1}$ (the series connection of $G_{j}$ and $H_j$), %with respect to $xy$ and $uv$,
  %is $3$-colourable if and only if $G_{j}$ is $3$-colourable, for each $j \in \{0,1,\dotsc,q-1\}$.  Since $G_0$ is not $3$-colourable, by \cref{main}, we deduce that 
 % $G$ is not $3$-colourable.
If a block of $G$ is a \fatras, then this block has chromatic number $4$ by Proposition~\ref{prop:fatras}(i), and thus $\chi(G)\geq 4$.

%\medskip

  Conversely, assume that no block of $G$ is a \fatras.
    We will show that $G$ is $3$-colourable by induction on the number of vertices.
  We may assume that $G$ is $2$-connected (since if each block is $3$-colourable, then it is straightforward to piece these $3$-colourings together to obtain a $3$-colouring of $G$).
 %
 %\smallskip
 Moreover, if $G$ is $3$-connected, then the result follows from \cref{main} since $G$ is not an odd wheel.
 Henceforth, we assume that $G$ is not $3$-connected.
 
 Let $(A,\{x,y\},B)$ be a $2$-separation of $V(G)$. %such that $A$ is minimal.
 Let $H_A$ (respectively, $H_B$) be the graph obtained from $G_A=G[A \cup \{x,y\}]$ (respectively, $G_B=G[B \cup \{x,y\}]$) by adding an edge $xy$ if it does not exist.
 Observe that since $G$ is $2$-connected,
 there is at least one $xy$-path in $G_B$, so
 $H_A$ (and, similarly, $H_B$)
  has maximal local edge-connectivity~$3$.

Assume first that neither $H_A$ nor $H_B$ are \fatras es.
 By the induction hypothesis, both $H_A$ and $H_B$ are $3$-colourable.
Thus, by piecing together a $3$-colouring of $H_A$ and a $3$-colouring of $H_B$ in both of which $x$ is coloured $1$ and $y$ is coloured $2$, we obtain a $3$-colouring of $G$.

%\medskip

Henceforth, we may assume that $H_A$ or $H_B$ is a \fatras.
Without loss of generality, we assume that $H_A$ is a \fatras.
Observe first that $xy \notin E(G)$.
Indeed, if $xy\in E(G)$, then $\lambda_{H_A}(x,y)\leq 2$, since there is an $xy$-path in $G_B\ba xy$, as $G$ is $2$-connected.
Hence, by Proposition~\ref{prop:fatras}(ii), $H_A$ is not a \fatras; a contradiction.

Furthermore, Proposition~\ref{prop:fatras}(ii) implies that there are three edge-disjoint $xy$-paths in $H_A$, two of which %, say $P_1$ and $P_2$,
are in $G_A$. 
Now, since $\lambda_G(x,y)\leq 3$, it follows that $\lambda_{G_B}(x,y)\leq 1$.
But $G_B$ is connected, since $G$ is $2$-connected, so there exists an edge $x'y'$ such that $G_B\ba x'y'$ has two components:
one, $G_x$, containing both $x$ and $x'$; and the other, $G_y$, containing $y$ and $y'$.
We now distinguish two cases depending on whether or not $x=x'$ or $y=y'$.

\begin{itemize}

\item Assume first that $x\neq x'$ and $y\neq y'$.
Let $H_x$ (respectively, $H_y$) be the graph obtained from $G_x$ (respectively, $G_y$) by adding the edge $xx'$ (respectively, $yy'$), if it does not exist.
Observe that the concatenation of an $xy$-path in $G_A$, a $yy'$-path in $G_y$, and $y'x'$ is a non-trivial $xx'$-path in $G$ whose internal vertices are not in $V(G_x)$.
Hence $\lambda_{G_x}(x,x')\leq 2$, so
$H_x$ %and, similarly, $H_y$
has maximal local edge-connectivity~$3$.
Moreover,
%Suppose first that $xx'\in E(G)$; so $G_x = H_x$.
%Since $\lambda_{G_x}(x,x')\leq 2$, %because there is an $xx'$-path in $G$ whose internal vertices are not in $V(G_x)$.
$G_x$ is not a \fatras, by \cref{prop:fatras}(ii), and hence $G_x$ is $3$-colourable, by the induction hypothesis.
Let $J$ be the graph obtained from $G-(V(G_x)\setminus \{x\})$ by adding the edge $xy'$.  Since there is an $xx'$-path in $G_x$, the graph $J$ has maximal local edge-connectivity~$3$.
Hence, by the induction hypothesis, either $J$ is $3$-colourable or $J$ is a \fatras.
In both cases, $G-(V(G_x)\setminus \{x\})$ is $3$-colourable, by Proposition~\ref{prop:fatras}(i).

Suppose that $xx'\in E(G)$. %; so $G_x = H_x$.
Then, in every $3$-colouring of $G_x$, the vertices $x$ and $x'$ have different colours. Consequently, one can find a $3$-colouring $c_1$ of $G_x$ and a $3$-colouring $c_2$ of $G-(V(G_x)\setminus \{x\})$ such that $c_1(x)=c_2(x)$ and $c_1(x') \neq c_2(y')$.
The union of these two colourings is a $3$-colouring of $G$.
Similarly, the result holds if $yy'\in E(G)$.

Henceforth, we may assume that $xx'$ and $yy'$ are not edges of $G$.
If both $H_x$ and $H_y$ are \fatras es, then $G$ is also a \fatras, obtained %from $H_A$, $H_x$ and $H_y$
by taking the Haj\'os join of $H_A$ and $H_x$ with respect to $(x,y)$ and $(x,x')$, and then the Haj\'os join of the resulting graph and $H_y$ with respect to $(y,x')$ and $(y,y')$.
Hence, we may assume that one of $H_x$ and $H_y$, say $H_x$, is not a \fatras.
Thus, by the induction hypothesis, $H_x$ admits a $3$-colouring $c_1$, which is a $3$-colouring of $G_x$ such that $c_1(x)\neq c_1(x')$.
%Now, as above, the graph $J$ obtained from $G-(V(G_x)\setminus \{x\})$ by adding the edge $xy'$ has maximal local edge-connectivity~$3$, and so
Since $G-(V(G_x)\setminus \{x\})$ is $3$-colourable, 
%Consequently,
one can find a $3$-colouring $c_2$ of $G-(V(G_x)\setminus \{x\})$ such that $c_1(x)=c_2(x)$ and $c_1(x') \neq c_2(y')$.
The union of $c_1$ and $c_2$ is a $3$-colouring of $G$.

\item Assume now that $x=x'$ or $y=y'$. Without loss of generality, $x=x'$.
Let $H_y$ be the graph obtained from $G_y$ by adding the edge $yy'$, if it does not exist.
The graph $H_y$ has maximal local edge-connectivity~$3$.
If $H_y$ is a \fatras, then $G$ is the Haj\'os join of $H_A$ and $H_y$ with respect to $(y,x)$ and $(y,y')$, so $G$ is also a \fatras; a contradiction.
If $H_y$ is not a \fatras, then by the induction hypothesis $H_y$ admits a $3$-colouring $c_2$, which is $3$-colouring of $G_y$ such that $c_2(y)\neq c_2(y')$.
Now $H_A$ is a \fatras, so it is $4$-critical by Proposition~\ref{prop:fatras}(i). Thus $G_A$ admits a $3$-colouring $c_1$ such that $c_1(x)=c_1(y)$.
Without loss of generality, we may assume that $c_1(y)=c_2(y)$. Then the union of $c_1$ and $c_2$ is a $3$-colouring of $G$. \qedhere
\end{itemize}
\end{proof}

\begin{corollary}
  \label{mainalgo}
  Let $G$ be a graph with maximal local edge-connectivity~$3$.
  Then there is a polynomial-time algorithm that finds an optimal colouring for $G$.
\end{corollary}

\section{Graphs with maximal local connectivity~$k$}
\label{secmlc3}

We now consider the more general class of graphs with maximal local (vertex) connectivity $k$.  %A natural question is: to what
%extent can the results of the previous section be extended to graphs in this class?
First, we show that for a $3$-connected graph, the notions of maximal local edge-connectivity~$3$ and maximal local connectivity~$3$ are equivalent.

\begin{lemma}
  \label{structurelemma1}
  Let $G$ be a $3$-connected graph with maximal local connectivity~$3$.
  Then $G$ has maximal local edge-connectivity~$3$.
\end{lemma}

\begin{proof}
  Consider two vertices $x$ and $y$ with four edge-disjoint paths between them.
  We will show that there is a pair of vertices with four internally disjoint paths between them, contradicting that $G$ has maximal local connectivity~$3$.
  First we assume that $x$ and $y$ are not adjacent.
  Let $(X,S,Y)$ be a $3$-separation with $x\in X$ and $y\in Y$ such
  that $X$ is inclusion-wise minimal. Let $S=\{v_1,v_2,v_3\}$; note that
  $3$-connectivity implies that every vertex in $S$ has a neighbour both
  in $X$ and $Y$. Each of the four paths has, when going from $x$ to
  $y$,  a last vertex in $X\cup S$. This vertex has to be in $S$, so we
  can assume, without loss of generality, that $v_1$ is the last
  such vertex of at least two of the four edge-disjoint paths. This
  means that $v_1$ has at least two neighbours in $Y$.

  We will show that there are four internally vertex-disjoint paths in
  $G[X\cup S]$: two $xv_1$-paths, an $xv_2$-path and an $xv_3$-path.
  Let $G'$ be the graph obtained from $G[X\cup S]$ by introducing a
  new vertex $v'_1$ that is adjacent to every neighbour of $v_1$ in $X\cup S$.
  If $G'$ contains four paths connecting $x$ and
  $S':=\{v_1,v'_1,v_2,v_3\}$ that meet only in $x$, then the required
  four paths exist in $G[X\cup S]$. If there are no four such paths in
  $G'$, then a max-flow min-cut argument (with $x$ having infinite
capacity and every other vertex having unit capacity) shows that
  there is a set $S^*$ of at most three vertices, with $x\not\in S^*$, that
  separate $x$ and $S'$. It is not possible that $S^*\subset S'$:
  then every vertex in the non-empty set $S'\setminus S^*$ remains
  reachable from $x$ (using that every vertex of $S'$ has a neighbour
  in $X$). Therefore, $S^*$ has at least one vertex in $X$ and hence
the set of vertices reachable from $x$ in $G' - S^*$ is a
  proper subset of $X$. It follows that $S^*$ implies the existence of
  a $3$-separation contradicting the minimality of $X$.

  Next we prove that there are internally disjoint $v_1v_2$- and $v_1v_3$-paths in $G[S\cup Y]$.
  Recall that
  $v_1$ has two neighbours in $Y$.
  Suppose, towards a contradiction, that given any $v_1v_2$-path and $v_1v_3$-path in $G[S\cup Y]$, these paths are not internally disjoint.
  Then, in $G[S \cup Y]$, there is a cut-vertex $w$ that separates $v_1$ and $\{v_2,v_3\}$.
  Since $v_1$ has two neighbours in $Y$, there is a vertex $q \in Y$ that is adjacent to $v_1$ and distinct from $w$.  As $w$ is a cut-vertex in $G[S \cup Y]$, every $qv_2$- or $qv_3$-path passes through $w$.  Hence $\{w,v_1\}$ separates $q$ from %$\{v_2,v_3\}$
  $x$
  in $G$, contradicting $3$-connectivity.

  Now there are internally disjoint $xv_1$-, $xv_1$-, $xv_2$-, and $xv_3$-paths in $X$ and internally disjoint $v_1v_2$- and $v_1v_3$-paths in $Y$.
  Thus, as shown \cref{fig:structureproof}, there are four internally disjoint $xv_1$-paths, contradicting the fact that the local connectivity $\kappa(x,v_1)$ is at most $3$.

  \begin{figure}
    \begin{subfigure}{0.48\textwidth}
      \centering
      \begin{tikzpicture}[scale=1]
        \tikzset{VertexStyle/.append style = {draw,minimum height=7,minimum width=7,fill=black}}
        %\SetVertexNoLabel
        \Vertex[x=-3,y=1,LabelOut=true,L=$x$,Lpos=180]{x}
        \Vertex[x=0,y=0,LabelOut=true,L=$v_3$,Lpos=-90]{z3}
        \Vertex[x=0,y=1,LabelOut=true,L=$v_2$,Lpos=-90]{z2}
        \Vertex[x=0,y=2,LabelOut=true,L=$v_1$,Lpos=90]{z1}
        %\Vertex[x=3,y=1]{y}

        \tikzset{EdgeStyle/.append style = {decorate,decoration={snake,amplitude=.3mm}}}
        %\tikzset{EdgeStyle/.append style = {dashed}}
        \tikzset{EdgeStyle/.append style = {bend right=10}}
        \Edge(x)(z2)
        \Edge(x)(z3)
        %\Edge(y)(z1)
        %\Edge(y)(z2)
        %\Edge(y)(z3)
        \Edge(x)(z1)
        \tikzset{EdgeStyle/.append style = {bend left=10}}
        \Edge(x)(z1)
        \tikzset{EdgeStyle/.append style = {bend left=60}}
        \Edge(z1)(z2)
        \tikzset{EdgeStyle/.append style = {bend left=90}}
        \Edge(z1)(z3)

      \end{tikzpicture}
      %\subcaption{The case where $x$ and $y$ are non-adjacent.}
      %\label{fig:sp1}
    \end{subfigure}
    \begin{subfigure}{0.48\textwidth}
      \centering
      \begin{tikzpicture}[scale=1]
        \tikzset{VertexStyle/.append style = {draw,minimum height=7,minimum width=7,fill=black}}
        %\SetVertexNoLabel
        \Vertex[x=-2,y=1.5,LabelOut=true,L=$x$,Lpos=180]{x}
        %\Vertex[x=0,y=0,LabelOut=true,L=$v_3$,Lpos=-90]{z3}
        \Vertex[x=0,y=0,LabelOut=true,L=$v_2$,Lpos=-90]{z2}
        \Vertex[x=0,y=1,LabelOut=true,L=$v_1$,Lpos=90]{z1}
        \Vertex[x=2,y=1.5,LabelOut=true,L=$y$]{y}

        \tikzset{EdgeStyle/.append style = {bend left=15}}
        \Edge(x)(y)
        \tikzset{EdgeStyle/.append style = {decorate,decoration={snake,amplitude=.3mm},bend left=0}}
        \Edge(z1)(y)
        \tikzset{EdgeStyle/.append style = {bend right=10}}
        \Edge(x)(z1)
        \Edge(x)(z2)
        \tikzset{EdgeStyle/.append style = {bend left=10}}
        \Edge(x)(z1)
        \tikzset{EdgeStyle/.append style = {bend left=60}}
        \Edge(z1)(z2)

      \end{tikzpicture}
      %\subcaption{The case where $x$ and $y$ are adjacent.}
      %\label{fig:sp2}
    \end{subfigure}
    \caption{The four internally disjoint $xv_1$-paths obtained in the proof of \cref{structurelemma1}, when $x$ and $y$ are non-adjacent (left) or adjacent (right). Wiggly lines represent internally disjoint paths.}
    \label{fig:structureproof}
  \end{figure}
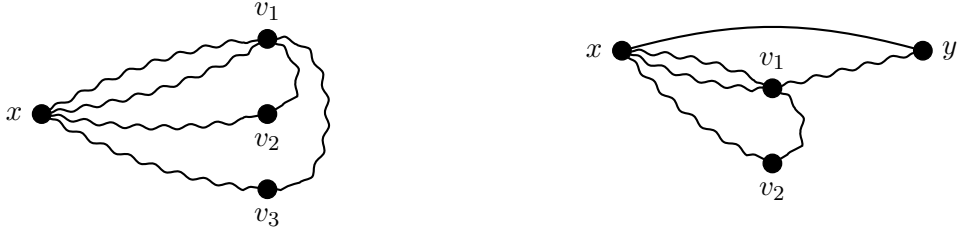

  A similar argument applies when $x$ and $y$ are adjacent.
  In this case, $G \ba xy$ has a $2$-vertex cut.  Let $(X,S,Y)$ be a $2$-separation of $G \ba xy$ with $x \in X$ and $y \in Y$ such that $X$ is inclusion-wise minimal, and let $S=\{v_1,v_2\}$.
  Since $G \ba xy$ is $2$-connected, $v_1$ and $v_2$ each have %both
  a neighbour in $X$ and a neighbour in $Y$.
  Each of the three $xy$-paths in $G \ba xy$ has a last vertex in $S$, so we may assume, without loss of generality, that $v_1$ is the last vertex of at least two of the three, and hence $v_1$ has at least two neighbours in $Y$.
  Let $G'$ be the graph obtained from $G[X \cup S]$ by introducing a new vertex $v_1'$ that is adjacent to every neighbour of $v_1$ in $X \cup S$, and let $S'= \{v_1,v_1',v_2\}$.
  If $G'$ does not contain three paths from $x$ to $S'$ that meet only in $x$, then, by a max-flow min-cut argument as in the case where $x$ and $y$ are not adjacent, we deduce there is a set $S^*$ of at most two vertices that separate $x$ and $S'$.  Since $S^* \not\subset S'$, this contradicts the minimality of $X$.

  It remains to prove that there are internally disjoint $v_1y$- and $v_1v_2$-paths in $G[Y\cup S]$.
  Suppose not.  Then, in $G[Y\cup S]$, there is a cut-vertex $w$ that separates $v_1$ and $\{v_2,y\}$.
  Since $v_1$ has at least two neighbours in $Y$, one of these neighbours $q$ is distinct from $w$.
  As every $qv_2$- or $qy$-path in $G[Y \cup S]$ passes through $w$, it follows that
  $\{w,v_1\}$ separates $q$ from %$\{v_2,y\}$
  $x$
  in $G$, contradicting $3$-connectivity.
  This completes the proof of \cref{structurelemma1}.
\end{proof}

At this juncture, we observe that the proof of \cref{structurelemma1} relies on properties specific to $3$-connected graphs with local connectivity~$3$.
For $k\geq 4$, a $k$-connected graph with maximal local connectivity~$k$ might not have maximal local edge-connectivity~$k$;
an example is given in \cref{fig:noedge5lc}.
In particular, in the proof of \cref{structurelemma1}, %in the case where $x$ and $y$ are non-adjacent,
the argument that there are internally disjoint $v_1v_2$- and $v_1v_3$-paths in $G[S \cup Y]$ would not extend to the existence of a $v_1v_4$-path, as $v_1$ might not even have more than two neighbors in $Y$.

\begin{figure}
  \centering
  \begin{tikzpicture}[scale=1]
    \tikzset{VertexStyle/.append style = {draw,minimum height=7,minimum width=7,fill=black}}
    \SetVertexNoLabel
    \Vertex[x=-3,y=1.5]{x}
    \Vertex[x=-1.5,y=2.25]{q}

    \Vertex[x=0,y=0]{z1}
    \Vertex[x=0,y=1]{z2}
    \Vertex[x=0,y=2]{z3}
    \Vertex[x=0,y=3]{z4}

    \Vertex[x=1.5,y=2.25]{p}
    \Vertex[x=3,y=1.5]{y}

    \Edge(x)(q)
    \Edge(x)(z1)
    \Edge(x)(z2)
    \Edge(x)(z3)
    \Edge(z1)(q)
    \Edge(z2)(q)
    \Edge(z4)(q)
    \Edge(z2)(z3)
    \Edge(z1)(p)
    \Edge(z3)(p)
    \Edge(z4)(p)
    \Edge(y)(p)
    \Edge(y)(z1)
    \Edge(y)(z2)
    \Edge(y)(z3)

    \tikzset{EdgeStyle/.append style = {bend left}}
    \Edge(x)(z4)
    \Edge(z4)(y)
  \end{tikzpicture}
  \caption{A $4$-connected graph with maximal local connectivity~$4$, but maximal local edge-connectivity~$5$.}
  \label{fig:noedge5lc}
\end{figure}
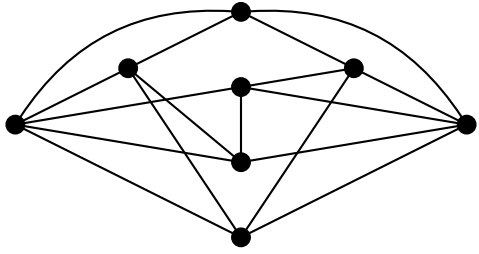

\Cref{colour_nl4c} now
follows immediately from \cref{main,algo,structurelemma1}.
%
%\begin{corollary}
  %Let $G$ be a $3$-connected graph with maximal local connectivity $3$.  Then $G$ is $3$-colourable if and only if $G$ is not $K_4$ or an odd wheel.  Moreover, there is a polynomial-time algorithm that finds an optimal colouring for $G$.
%\end{corollary}
%
One might hope to generalise this result to all graphs with maximal local connectivity~$3$, for a result analogous to \cref{thm:brooks3gen}. % for graphs with maximal local edge-connectivity $3$.
But this hope will not be realised, unless P=NP, since deciding if a $2$-connected graph with maximal local connectivity~$3$ is $3$-colourable is NP-complete.
%
%For graphs with maximal local connectivity $3$, it is only possible to characterise those that are $3$-colourable when we restrict our attention to the $3$-connected graphs, unless $P=NP$, by the next result.
%
We %will
prove this
using a reduction from the unrestricted version of \textsc{$3$-colourability}.
Given an instance of this problem,
we replace each vertex of degree at least four with a gadget that ensures that the resulting graph has maximal local connectivity~$3$.
Shortly, we describe this gadget; first, we require some definitions.

We call the graph obtained from two copies of a diamond, by identifying a pick vertex
%degree-$2$ vertex
from each, a \emph{serial diamond pair} and denote it $D_2$.
We call the two degree-$2$ vertices of $D_2$ the \emph{ends}. 
A tree is \emph{cubic} if all vertices have either degree one or degree three.  A degree-$1$ vertex is a \emph{leaf}; and an edge that is incident to a leaf is a \emph{pendant} edge, whereas an edge that is incident to two degree-$3$ vertices is an \emph{internal} edge.

For $l \geq 4$, let $T$ be a cubic tree with $l$ leaves. 
For each pendant edge $xy$, we remove $xy$, take a copy of a diamond $D$ and identify, firstly, the vertex $x$ with one pick vertex of $D$, and, secondly, $y$ with the other pick vertex of $D$.
For each internal edge $xy$, we remove $xy$, take a copy of $D_2$ and identify, firstly, the vertex $x$ with one end of $D_2$, and, secondly, $y$ with the other end of $D_2$.
A degree-$2$ vertex in the resulting graph $T'$ corresponds to a leaf of $T$; we call such a vertex an \emph{outlet}.
We also call $T'$ a \emph{hub gadget} with $l$ outlets. 
Observe that for any integer $l \ge 4$, there exists a hub gadget with exactly $l$ outlets. 
When $T'$ is used to replace a vertex $h$, we say $T'$ is the \emph{hub gadget of $h$}.
An example of a hub gadget with four outlets is shown in \cref{outlets-fig}.

\begin{figure}
  \centering
  \begin{tikzpicture}[scale=1.2]
    \tikzset{VertexStyle/.append style = {draw,minimum height=7,minimum width=7}}
    \Vertex[x=-1.875,y=0.9,LabelOut=true,L=$p_1$,Lpos=180]{a1}
    \Vertex[x=-1.875,y=-0.9,LabelOut=true,L=$p_2$,Lpos=180]{a2}
    \Vertex[x=1.875,y=0.9,LabelOut=true,L=$p_3$]{b3}
    \Vertex[x=1.875,y=-0.9,LabelOut=true,L=$p_4$]{b4}
    \SetVertexNoLabel
    \tikzset{VertexStyle/.append style = {fill=black}}
    \Vertex[x=-1.25,y=0]{a}
    \Vertex[x=-1.25,y=0.6]{aa1}
    \Vertex[x=-1.875,y=0.3]{aa2}

    \Vertex[x=-1.25,y=-0.6]{aa5}
    \Vertex[x=-1.875,y=-0.3]{aa6}

    \Vertex[x=1.25,y=0.6]{bb1}
    \Vertex[x=1.875,y=0.3]{bb2}

    \Vertex[x=1.25,y=-0.6]{bb5}
    \Vertex[x=1.875,y=-0.3]{bb6}

    \Vertex[x=-0.625,y=0.3]{dab1}
    \Vertex[x=-0.625,y=-0.3]{dab2}
    \Vertex[x=0.625,y=0.3]{dab3}
    \Vertex[x=0.625,y=-0.3]{dab4}
    \Vertex[x=0,y=0]{ab}
    \Vertex[x=1.25,y=0]{b}

    \Edge(a)(aa1)
    \Edge(a)(aa2)
    \Edge(aa1)(aa2)
    \Edge(a1)(aa2)
    \Edge(a1)(aa1)

    \Edge(a)(aa5)
    \Edge(a)(aa6)
    \Edge(aa5)(aa6)
    \Edge(a2)(aa6)
    \Edge(a2)(aa5)

    \Edge(a)(dab1)
    \Edge(a)(dab2)
    \Edge(dab1)(dab2)
    \Edge(ab)(dab1)
    \Edge(ab)(dab2)
    \Edge(ab)(dab3)
    \Edge(ab)(dab4)
    \Edge(dab3)(dab4)
    \Edge(b)(dab3)
    \Edge(b)(dab4)

    \Edge(b)(bb1)
    \Edge(b)(bb2)
    \Edge(bb1)(bb2)
    \Edge(b3)(bb2)
    \Edge(b3)(bb1)

    \Edge(b)(bb5)
    \Edge(b)(bb6)
    \Edge(bb5)(bb6)
    \Edge(b4)(bb6)
    \Edge(b4)(bb5)
  \end{tikzpicture}
  \caption{A hub gadget with four outlets $p_1$, $p_2$, $p_3$ and $p_4$.}
  \label{outlets-fig}
\end{figure}
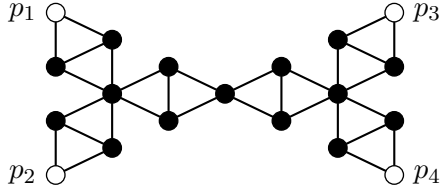

\begin{proposition}
  \label{mlcnpc3}
  The problem of deciding if a $2$-connected graph with maximal local connectivity~$3$ is $3$-colourable is NP-complete.
\end{proposition}
\begin{proof}
  Let $G$ be an instance of \textsc{$3$-colourability}.
  We may assume that $G$ is $2$-connected.
  For each $v \in V(G)$ such that $d(v) \geq 4$,
  we replace $v$ with a hub gadget with outlets $p_1,p_2,\dotsc,p_{d(v)}$, such that
  each neighbour $n_i$ of $v$ in $G$ is adjacent to $p_i$, for $i \in \seq{d(v)}$.
  Thus each outlet has degree three in the resulting graph $G'$.

  It is clear that $G'$ is $2$-connected.
  Now we show that $G'$ has maximal local connectivity~$3$.
  Clearly $\lc(x,y) \leq 3$ if $d(x) \leq 3$ or $d(y) \leq 3$.
  Suppose $d(x),d(y) \geq 4$. Then $x$ and $y$ belong to a hub gadget and are not outlets.  So $x$ belongs to either two or three diamonds, each with a pick vertex distinct from $x$. Let $P$ be the set of these pick vertices. When $y \notin P$, an $xy$-path must pass through some $p \in P$, so $\lc(x,y) \leq 3$ as required.  Otherwise, $x$ and $y$ are pick vertices of a diamond $D$, and there are two internally vertex disjoint $xy$-paths in $D$.  But $D$ is contained in a serial diamond pair $D_2$, and all other $xy$-paths must pass through the end of $D_2$ distinct from $x$ and $y$.  So $\lc(x,y) \leq 3$, as required.

  Suppose $G$ is $3$-colourable and let $\phi$ be a $3$-colouring of $G$.  
  We show that $G'$ is $3$-colourable.  
  Start by colouring each vertex $v$ in $V(G)\cap V(G')$ the colour $\phi(v)$. 
  For each hub gadget $H$ of $G'$ corresponding to a vertex $h$ of $G$, colour every pick vertex of a diamond in $H$ the colour $\phi(h)$.  Clearly, each outlet is given a different colour to its neighbours in $V(G)$ since $\phi$ is a $3$-colouring of $G$.
  The remaining two vertices of each diamond contained in $H$ have two neighbours the same colour $\phi(h)$, so can be coloured using the other two available colours.  Thus $G'$ is $3$-colourable.

  Now suppose that $G'$ is $3$-colourable. Each pick vertex of a diamond must have the same colour in a $3$-colouring of $G'$, so all outlets of a hub gadget have the same colour.  Let $H$ be the hub gadget of $h$, where $h \in V(G)$.  We colour $h$ with the colour of all the outlets of $H$ in the $3$-colouring of $G'$.  For each vertex $v \in V(G) \cap V(G')$, we colour $v$ with the same colour as in the $3$-colouring of $G'$, thus obtaining a $3$-colouring of $G$.
\end{proof}

A similar approach can be used to show that \textsc{$3$-colourability} remains NP-complete for $(k-1)$-connected graphs with maximal local edge-connectivity~$k$, for any $k \geq 4$. % (\cref{colour_n4c_2}).
To prove this, we first require the following lemma:

\begin{lemma}
  \label{conn-lemma}
  Let $k \geq 3$ and $j \geq 1$.
  Then
  \textsc{$k$-colourability} remains NP-complete when restricted to $j$-connected graphs.
\end{lemma}
\begin{proof}
  We show that \textsc{$k$-colourability} restricted to $j$-connected graphs is reducible to \textsc{$k$-colourability} restricted to $(j+1)$-connected graphs, for any fixed $j \geq 1$.
  Let $G_0$ be a $j$-connected graph; we construct a $(j+1)$-connected graph $G'$ such that $G_0$ is $k$-colourable if and only if $G'$ is.
  Let $S_0$ be a $j$-vertex cut in $G_0$, let $s \in S_0$, and let $G_1$ be the graph obtained from $G_0$ by introducing a single vertex $s'$ with the same neighbourhood as $s$.
  Now if $S'$ is a $j'$-vertex cut in $G_1$, for $j' \leq j$, then $S'$, or $S' \setminus \{s'\}$, is a $j'$-vertex cut, or $(j'-1)$-vertex cut, in $G_0$.
  Since $S_0$ is not a $j$-vertex cut in $G_1$, it follows that
  $G_1$ has strictly fewer $j$-vertex cuts than $G_0$.
  Repeat this process for each $j$-vertex cut $S_i$ in $G_i$ (there are polynomially many), and let $G'$ be the resulting graph.
  Then $G'$ has no vertex cuts of size at most $j$, so $G'$ is $(j+1)$-connected.  Moreover,
it is straightforward to verify that $G'$ is $k$-colourable if and only if $G_0$ is $k$-colourable.
\end{proof}

%In order to prove \cref{colour_n4c_2} for $k \geq 4$, we 
  We %will
  perform a reduction from \textsc{$k$-colourability} restricted to $(k-1)$-connected graphs (which is NP-complete by \cref{conn-lemma}).
  Let $G$ be a $(k-1)$-connected graph.
  %We will construct a $(k-1)$-connected graph $G'$ with maximal local connectivity $k$, such that $G$ is $k$-colourable if and only if $G'$ is $k$-colourable.
%
  For each vertex $v$ with $d(v) \geq k+1$, we will ``replace'' it with a gadget in such a way that the resulting graph $G'$ remains $(k-1)$-connected, $G'$ is $3$-colourable if and only if $G$ is $3$-colourable, and no vertex of $G'$ has degree greater than $k$.
  %We say that a gadget for replacing a vertex of degree $d$ is a $(d,k)$-gadget.

  We will describe, momentarily, a gadget $G_{l,k}$ used to replace a vertex $v$ of degree $l$, where $l > k$, with
  vertices $x_1,x_2,\dotsc,x_l \in V(G_{l,k})$ called the \emph{outlets} of $G_{l,k}$.
  Let $G_{l,k}$ be a gadget, and let $G$ be a graph with a vertex $v$ of degree $l > k$.
  We say that we \emph{attach $G_{l,k}$ to $G$ at $v$} when we perform the following operation: relabel the vertices of $G$ such that
  %$N_G(v) = \{v_1,v_2,\dotsc,v_l\}$ and $V(G) \cap V(G_{l,k}) = N_G(v)$,
  $V(G) \cap V(G_{l,k}) = N_G(v) = \{x_1,x_2,\dotsc,x_l\}$,
  and construct the graph $(G \cup G_{l,k}) - \{v\}$.

  We now give a recursive description of $G_{l,k}$.
  %we take the 
  %We describe this gadget recursively.
  %First, let $v$ be a vertex of degree $l$, with $k+1 \leq l \leq (k-2)(k-1)$. 
  First, suppose that $l \leq (k-2)(k-1)$.
  %In this case, our gadget is a complete bipartite graph, plus an extra edge.
  %We %replace $v$ with a gadget constructed
  %construct this gadget
  %as follows.
%
  Let $a = \lceil{l / (k-1)}\rceil$, and
  let $(B_1, B_2, \dotsc, B_{k-1})$ be a partition of $\{x_1,x_2,\dotsc,x_l\}$ into $k-1$ cells of size $a-1$ or $a$.
  %Note that the size of each set is at least one and at most $k-2$. % (when $l = (k-2)(k-1)$, all $k-1$ sets are of size $a$).
  We construct $G_{l,k}$ starting from
  %We start with
  a copy of the complete bipartite graph $K_{k-1,k-a}$
  %Let $K_{k-1,k-a}$ be the complete bipartite graph
  where the vertices of the $(k-1)$-vertex partite set
  are labelled $b_1, b_2, \dotsc, b_{k-1}$,
  and the remaining vertices are labelled $u_1, u_2, \dotsc, u_{k-a}$.
  Since $k \geq 4$ and $2 \leq a \leq k-2$, we have $k-a \geq 2$.
  Add an edge $u_1u_2$, and
  %Identify distinct vertices $v_1,v_2$ in $V$ (which exist, since $k \geq 4$, and $2 \leq a \leq k-2$), and introduce an edge $v_1v_2$.
  for each $i \in \seq{k-1}$ and $w \in B_i$, add an edge $wb_i$.
  %For each vertex labelled $B_i$ for some $i \in \seq{k-1}$, and for each $l \in B_i$, we introduce an edge $xb_i$.
  We call the resulting graph $G_{l,k}$ and it is illustrated in \cref{gadgetdiagrams}(a).

  Now suppose $l > (k-2)(k-1)$.
  Let $(B_1, B_2, \dotsc, B_{k-1})$ be a partition of $\{x_1,x_2,\dotsc,x_l\}$ such that $|B_i| = k-2$ for $i \in \seq{k-2}$, and $|B_{k-1}| > k-2$.
  Take a copy of $K_{k-1,1,1}$, labelling the vertices of the $(k-1)$-vertex partite set as $b_1, b_2, \dotsc, b_{k-1}$, and the other two vertices $u_1$ and $u_2$.
  For each $i \in \seq{k-1}$, and for each $w \in B_i$, we introduce an edge $wb_i$.
  %We call the resulting graph $G_{l,k}^-$.
  Label the resulting graph $H_{l,k}$;
  we call $H_{l,k}$ an \emph{intermediate gadget} (see \cref{gadgetdiagrams}(b)).
  Let $l_1 = d_{H_{l,k}}(b_{k-1})$. % = |B_{k-1}|+2$.
  Since $l_1 = l - (k-2)^2+2$, we have $k+1 \leq l_1 \leq l-2$.
  The graph $G_{l,k}$ is obtained by attaching $G_{l_1,k}$ to $H_{l,k}$ at $b_{k-1}$.
  An example of such a gadget, for $l=10, k=4$, is given in \cref{gadgetexample}, and the intermediate gadgets involved in its construction are given in \cref{gadgetexampleconstruction}.

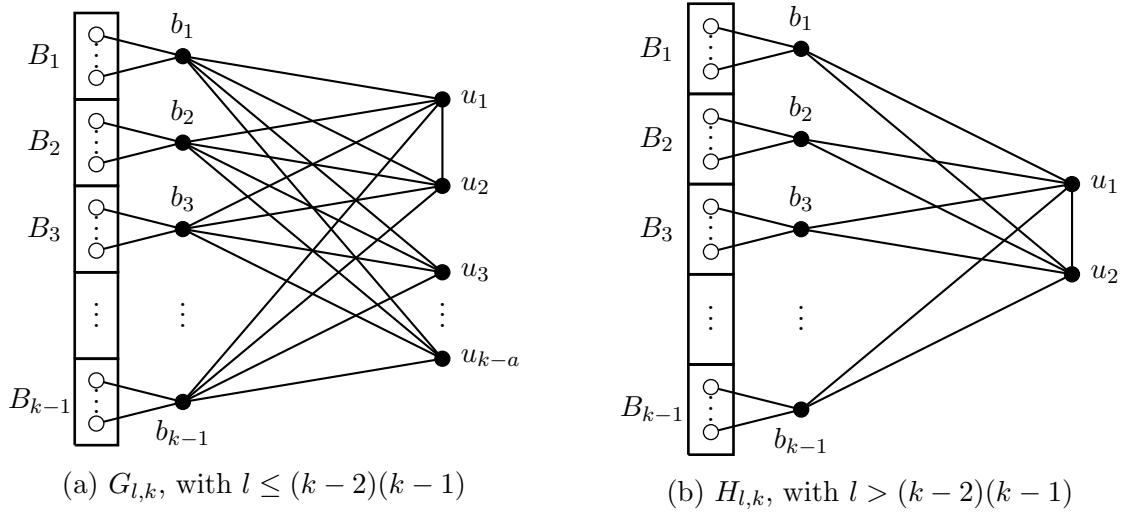
\begin{figure}
  \centering
  \begin{subfigure}{0.48\textwidth}
    \centering
    \begin{tikzpicture}[scale=1.15,line width=1pt]
      \SetVertexNoLabel
      \tikzset{VertexStyle/.append style = {fill=white,minimum height=1,minimum width=1}}
      \tikzset{VertexStyle/.append style = {draw}}
      \Vertex[x=0,y=4.25,LabelOut=true,Math=true,Lpos=180,L=x_1]{x1}
      \Vertex[x=0,y=3.75,LabelOut=true,Math=true,Lpos=180,L=x_2]{x2}
      \Vertex[x=0,y=3.25,LabelOut=true,Math=true,Lpos=180,L=x_3]{x3}
      \Vertex[x=0,y=2.75,LabelOut=true,Math=true,Lpos=180,L=x_4]{x4}
      \Vertex[x=0,y=2.25,LabelOut=true,Math=true,Lpos=180,L=x_5]{x5}
      \Vertex[x=0,y=1.75,LabelOut=true,Math=true,Lpos=180,L=x_6]{x6}
      \Vertex[x=0,y=-0.25,LabelOut=true,Math=true,Lpos=180,L=x_5]{xk1}
      \Vertex[x=0,y=0.25,LabelOut=true,Math=true,Lpos=180,L=x_k]{xk}
      \node at (0,0.1) {$\vdots$};
      \node at (0,2.1) {$\vdots$};
      \node at (0,3.1) {$\vdots$};
      \node at (0,4.1) {$\vdots$};
      \draw (-.25,-.5) -- (-.25,0.5) -- (.25,.5) -- (0.25,-.5) -- (-.25,-.5);
      \draw (-.25,0.5) -- (-.25,1.5) -- (.25,1.5) -- (0.25,.5) -- (-.25,0.5);
      \draw (-.25,1.5) -- (-.25,2.5) -- (.25,2.5) -- (.25,1.5) -- (-.25,1.5);
      \draw (-.25,2.5) -- (-.25,3.5) -- (.25,3.5) -- (.25,2.5) -- (-.25,2.5);
      \draw (-.25,3.5) -- (-.25,4.5) -- (.25,4.5) -- (.25,3.5) -- (-.25,3.5);
      \node at (0,1.1) {$\vdots$};
      \node at (-0.65,0) {$B_{k-1}$};
      \node at (-0.6,2) {$B_3$};
      \node at (-0.6,3) {$B_2$};
      \node at (-0.6,4) {$B_1$};
      \tikzset{VertexStyle/.append style = {fill=black,minimum height=4,minimum width=4}}
      \SetVertexLabel
      \Vertex[x=1,y=4,LabelOut=true,Math=true,Lpos=90,L=b_1]{b1}
      \Vertex[x=1,y=3,LabelOut=true,Math=true,Lpos=90,L=b_2]{b2}
      \Vertex[x=1,y=2,LabelOut=true,Math=true,Lpos=90,L=b_3]{b3}
      \node at (1,1.1) {$\vdots$};
      \Vertex[x=1,y=0,LabelOut=true,Math=true,Lpos=-90,L=b_{k-1}]{bk}
      \Vertex[x=4,y=3.5,LabelOut=true,Math=true,Lpos=0,L=u_1]{u1}
      \Vertex[x=4,y=2.5,LabelOut=true,Math=true,Lpos=0,L=u_2]{u2}
      \Vertex[x=4,y=1.5,LabelOut=true,Math=true,Lpos=0,L=u_3]{u3}
      \node at (4,1.1) {$\vdots$};
      \Vertex[x=4,y=0.5,LabelOut=true,Math=true,Lpos=0,L=u_{k-a}]{uka}
      \Edge(x1)(b1)
      \Edge(x2)(b1)
      \Edge(x3)(b2)
      \Edge(x4)(b2)
      \Edge(x5)(b3)
      \Edge(x6)(b3)
      \Edge(xk1)(bk)
      \Edge(xk)(bk)
      \Edge(u1)(b1)
      \Edge(u2)(b1)
      \Edge(u1)(b2)
      \Edge(u2)(b2)
      \Edge(u1)(b3)
      \Edge(u2)(b3)
      \Edge(u1)(bk)
      \Edge(u2)(bk)
      \Edge(u1)(u2)
      \Edge(u3)(b1)
      \Edge(uka)(b1)
      \Edge(u3)(b2)
      \Edge(uka)(b2)
      \Edge(u3)(b3)
      \Edge(uka)(b3)
      \Edge(u3)(bk)
      \Edge(uka)(bk)
    \end{tikzpicture} \\
    (a) $G_{l,k}$, with $l \le (k-2)(k-1)$
  \end{subfigure}
  \begin{subfigure}{0.48\textwidth}
    \centering
    \begin{tikzpicture}[scale=1.2,line width=1pt]
      \SetVertexNoLabel
      \tikzset{VertexStyle/.append style = {fill=white,minimum height=1,minimum width=1}}
      \tikzset{VertexStyle/.append style = {draw}}
      \Vertex[x=0,y=4.25,LabelOut=true,Math=true,Lpos=180,L=x_1]{x1}
      \Vertex[x=0,y=3.75,LabelOut=true,Math=true,Lpos=180,L=x_2]{x2}
      \Vertex[x=0,y=3.25,LabelOut=true,Math=true,Lpos=180,L=x_3]{x3}
      \Vertex[x=0,y=2.75,LabelOut=true,Math=true,Lpos=180,L=x_4]{x4}
      \Vertex[x=0,y=2.25,LabelOut=true,Math=true,Lpos=180,L=x_5]{x5}
      \Vertex[x=0,y=1.75,LabelOut=true,Math=true,Lpos=180,L=x_6]{x6}
      \Vertex[x=0,y=-0.25,LabelOut=true,Math=true,Lpos=180,L=x_5]{xk1}
      \Vertex[x=0,y=0.25,LabelOut=true,Math=true,Lpos=180,L=x_k]{xk}
      \node at (0,0.1) {$\vdots$};
      \node at (0,2.1) {$\vdots$};
      \node at (0,3.1) {$\vdots$};
      \node at (0,4.1) {$\vdots$};
      \draw (-.25,-.5) -- (-.25,0.5) -- (.25,.5) -- (0.25,-.5) -- (-.25,-.5);
      \draw (-.25,0.5) -- (-.25,1.5) -- (.25,1.5) -- (0.25,.5) -- (-.25,0.5);
      \draw (-.25,1.5) -- (-.25,2.5) -- (.25,2.5) -- (.25,1.5) -- (-.25,1.5);
      \draw (-.25,2.5) -- (-.25,3.5) -- (.25,3.5) -- (.25,2.5) -- (-.25,2.5);
      \draw (-.25,3.5) -- (-.25,4.5) -- (.25,4.5) -- (.25,3.5) -- (-.25,3.5);
      \node at (0,1.1) {$\vdots$};
      \node at (-0.65,0) {$B_{k-1}$};
      \node at (-0.6,2) {$B_3$};
      \node at (-0.6,3) {$B_2$};
      \node at (-0.6,4) {$B_1$};
      \tikzset{VertexStyle/.append style = {fill=black,minimum height=4,minimum width=4}}
      \SetVertexLabel
      \Vertex[x=1,y=4,LabelOut=true,Math=true,Lpos=90,L=b_1]{b1}
      \Vertex[x=1,y=3,LabelOut=true,Math=true,Lpos=90,L=b_2]{b2}
      \Vertex[x=1,y=2,LabelOut=true,Math=true,Lpos=90,L=b_3]{b3}
      \node at (1,1.1) {$\vdots$};
      \Vertex[x=1,y=0,LabelOut=true,Math=true,Lpos=-90,L=b_{k-1}]{bk}
      \Vertex[x=4,y=2.5,LabelOut=true,Math=true,Lpos=0,L=u_1]{u1}
      \Vertex[x=4,y=1.5,LabelOut=true,Math=true,Lpos=0,L=u_2]{u2}
      \Edge(x1)(b1)
      \Edge(x2)(b1)
      \Edge(x3)(b2)
      \Edge(x4)(b2)
      \Edge(x5)(b3)
      \Edge(x6)(b3)
      \Edge(xk1)(bk)
      \Edge(xk)(bk)
      \Edge(u1)(b1)
      \Edge(u2)(b1)
      \Edge(u1)(b2)
      \Edge(u2)(b2)
      \Edge(u1)(b3)
      \Edge(u2)(b3)
      \Edge(u1)(bk)
      \Edge(u2)(bk)
      \Edge(u1)(u2)
    \end{tikzpicture}\\
    (b) $H_{l,k}$, with $l > (k-2)(k-1)$
  \end{subfigure}
  \caption{Gadgets and intermediate gadgets used in the proof of \cref{colour_n4c_2'}.}
  \label{gadgetdiagrams}
\end{figure}

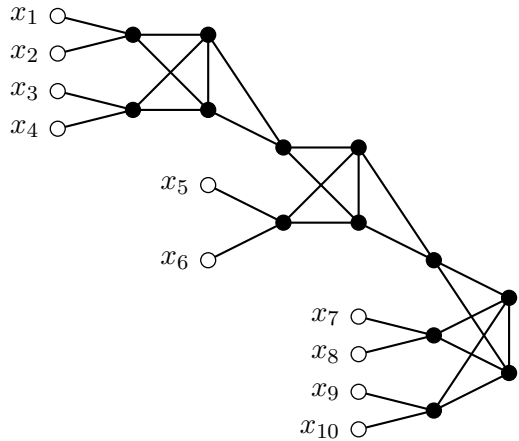
\begin{figure}
    \centering
    \begin{tikzpicture}[scale=1.0,line width=1pt]
      \tikzset{VertexStyle/.append style = {fill=white,draw,minimum height=4,minimum width=4}}
      \Vertex[x=0,y=6.25,LabelOut=true,Math=true,Lpos=180,L=x_1]{x1}
      \Vertex[x=0,y=5.75,LabelOut=true,Math=true,Lpos=180,L=x_2]{x2}
      \Vertex[x=0,y=5.25,LabelOut=true,Math=true,Lpos=180,L=x_3]{x3}
      \Vertex[x=0,y=4.75,LabelOut=true,Math=true,Lpos=180,L=x_4]{x4}
      \Vertex[x=2,y=4.0,LabelOut=true,Math=true,Lpos=180,L=x_5]{x5}
      \Vertex[x=2,y=3.0,LabelOut=true,Math=true,Lpos=180,L=x_6]{x6}
      \Vertex[x=4,y=2.25,LabelOut=true,Math=true,Lpos=180,L=x_7]{x7}
      \Vertex[x=4,y=1.75,LabelOut=true,Math=true,Lpos=180,L=x_8]{x8}
      \Vertex[x=4,y=1.25,LabelOut=true,Math=true,Lpos=180,L=x_9]{x9}
      \Vertex[x=4,y=0.75,LabelOut=true,Math=true,Lpos=180,L=x_{10}]{x10}
      \tikzset{VertexStyle/.append style = {fill=black}}
      \SetVertexNoLabel
      \Vertex[x=1,y=6,LabelOut=true,Math=true,Lpos=90]{b1}
      \Vertex[x=1,y=5,LabelOut=true,Math=true,Lpos=90]{b2}
      \Vertex[x=2,y=6,LabelOut=true,Math=true,Lpos=90]{u1}
      \Vertex[x=2,y=5,LabelOut=true,Math=true,Lpos=90]{u2}
      \Vertex[x=3,y=4.5,LabelOut=true,Math=true,Lpos=90]{b3}
      \Vertex[x=3,y=3.5,LabelOut=true,Math=true,Lpos=90]{b4}
      \Vertex[x=4,y=4.5,LabelOut=true,Math=true,Lpos=90]{u3}
      \Vertex[x=4,y=3.5,LabelOut=true,Math=true,Lpos=90]{u4}
      \Vertex[x=5,y=3,LabelOut=true,Math=true,Lpos=90]{b5}
      \Vertex[x=5,y=2,LabelOut=true,Math=true,Lpos=90]{b6}
      \Vertex[x=5,y=1,LabelOut=true,Math=true,Lpos=90]{b7}
      \Vertex[x=6,y=2.5,LabelOut=true,Math=true,Lpos=90]{u5}
      \Vertex[x=6,y=1.5,LabelOut=true,Math=true,Lpos=90]{u6}
      \Edge(x1)(b1)
      \Edge(x2)(b1)
      \Edge(x3)(b2)
      \Edge(x4)(b2)
      \Edge(u1)(b1)
      \Edge(u2)(b1)
      \Edge(u1)(b2)
      \Edge(u2)(b2)
      \Edge(u1)(u2)
      \Edge(u1)(b3)
      \Edge(u2)(b3)
      \Edge(u3)(u4)
      \Edge(u3)(b4)
      \Edge(b4)(u4)
      \Edge(u3)(b3)
      \Edge(b3)(u4)
      \Edge(b4)(x5)
      \Edge(b4)(x6)
      \Edge(b5)(u3)
      \Edge(b5)(u4)
      \Edge(b6)(x7)
      \Edge(b6)(x8)
      \Edge(b7)(x9)
      \Edge(b7)(x10)
      \Edge(u5)(b5)
      \Edge(u6)(b5)
      \Edge(u5)(b6)
      \Edge(u6)(b6)
      \Edge(u5)(b7)
      \Edge(u6)(b7)
      \Edge(u5)(u6)
    \end{tikzpicture}
    \caption{An example of a gadget, $G_{10,4}$.}
  \label{gadgetexample}
\end{figure}

\begin{figure}
  \centering
  \begin{subfigure}{0.35\textwidth}
    \centering
    \begin{tikzpicture}[scale=0.9,line width=1pt]
      \tikzset{VertexStyle/.append style = {fill=white,draw,minimum height=4,minimum width=4}}
      \Vertex[x=0,y=6.25,LabelOut=true,Math=true,Lpos=180,L=x_1]{x1}
      \Vertex[x=0,y=5.75,LabelOut=true,Math=true,Lpos=180,L=x_2]{x2}
      \Vertex[x=0,y=5.25,LabelOut=true,Math=true,Lpos=180,L=x_3]{x3}
      \Vertex[x=0,y=4.75,LabelOut=true,Math=true,Lpos=180,L=x_4]{x4}
      \Vertex[x=4,y=4.75,LabelOut=true,Math=true,L=x_5]{x5}
      \Vertex[x=4,y=4.25,LabelOut=true,Math=true,L=x_6]{x6}
      \Vertex[x=4,y=3.75,LabelOut=true,Math=true,L=x_7]{x7}
      \Vertex[x=4,y=3.25,LabelOut=true,Math=true,L=x_8]{x8}
      \Vertex[x=4,y=2.75,LabelOut=true,Math=true,L=x_9]{x9}
      \Vertex[x=4,y=2.25,LabelOut=true,Math=true,L=x_{10}]{x10}
      \tikzset{VertexStyle/.append style = {fill=black}}
      \Vertex[x=2,y=6,LabelOut=true,Math=true,Lpos=90,L=u_1]{u1}
      \Vertex[x=2,y=5,LabelOut=true,Math=true,Lpos=-90,L=u_2]{u2}
      \SetVertexNoLabel
      \Vertex[x=1,y=6,LabelOut=true,Math=true,Lpos=90]{b1}
      \Vertex[x=1,y=5,LabelOut=true,Math=true,Lpos=90]{b2}
      \Vertex[x=3,y=4.5,LabelOut=true,Math=true,Lpos=90]{b3}
      \Edge(x1)(b1)
      \Edge(x2)(b1)
      \Edge(x3)(b2)
      \Edge(x4)(b2)
      \Edge(u1)(b1)
      \Edge(u2)(b1)
      \Edge(u1)(b2)
      \Edge(u2)(b2)
      \Edge(u1)(u2)
      \Edge(u1)(b3)
      \Edge(u2)(b3)
      \Edge(b3)(x5)
      \Edge(b3)(x6)
      \Edge(b3)(x7)
      \Edge(b3)(x8)
      \Edge(b3)(x9)
      \Edge(b3)(x10)
    \end{tikzpicture} \\
    (a) $H_{10,4}$
  \end{subfigure}
  \begin{subfigure}{0.35\textwidth}
    \centering
    \begin{tikzpicture}[scale=0.9,line width=1pt]
      \tikzset{VertexStyle/.append style = {fill=white,draw,minimum height=4,minimum width=4}}
      \Vertex[x=2,y=4.0,LabelOut=true,Math=true,Lpos=180,L=x_5]{x5}
      \Vertex[x=2,y=3.0,LabelOut=true,Math=true,Lpos=180,L=x_6]{x6}
      \Vertex[x=6,y=3.25,LabelOut=true,Math=true,L=x_7]{x7}
      \Vertex[x=6,y=2.75,LabelOut=true,Math=true,L=x_8]{x8}
      \Vertex[x=6,y=2.25,LabelOut=true,Math=true,L=x_9]{x9}
      \Vertex[x=6,y=1.75,LabelOut=true,Math=true,L=x_{10}]{x10}
      \tikzset{VertexStyle/.append style = {fill=black}}
      \Vertex[x=2,y=6,LabelOut=true,Math=true,Lpos=180,L=u_1]{u1}
      \Vertex[x=2,y=5,LabelOut=true,Math=true,Lpos=180,L=u_2]{u2}
      \Vertex[x=4,y=4.5,LabelOut=true,Math=true,Lpos=90,L=u_1']{u3}
      \Vertex[x=4,y=3.5,LabelOut=true,Math=true,Lpos=-90,L=u_2']{u4}
      \SetVertexNoLabel
      \Vertex[x=3,y=4.5,LabelOut=true,Math=true,Lpos=90]{b3}
      \Vertex[x=3,y=3.5,LabelOut=true,Math=true,Lpos=90]{b4}
      \Vertex[x=5,y=3,LabelOut=true,Math=true,Lpos=90]{b5}
      \Edge(u1)(b3)
      \Edge(u2)(b3)
      \Edge(u3)(u4)
      \Edge(u3)(b4)
      \Edge(b4)(u4)
      \Edge(u3)(b3)
      \Edge(b3)(u4)
      \Edge(b4)(x5)
      \Edge(b4)(x6)
      \Edge(b5)(u3)
      \Edge(b5)(u4)
      \Edge(b5)(x7)
      \Edge(b5)(x8)
      \Edge(b5)(x9)
      \Edge(b5)(x10)
    \end{tikzpicture} \\
    (b) $H_{8,4}$
  \end{subfigure}
  \begin{subfigure}{0.25\textwidth}
    \centering
    \begin{tikzpicture}[scale=0.9,line width=1pt]
      \tikzset{VertexStyle/.append style = {fill=white,draw,minimum height=4,minimum width=4}}
      \Vertex[x=4,y=2.25,LabelOut=true,Math=true,Lpos=180,L=x_7]{x7}
      \Vertex[x=4,y=1.75,LabelOut=true,Math=true,Lpos=180,L=x_8]{x8}
      \Vertex[x=4,y=1.25,LabelOut=true,Math=true,Lpos=180,L=x_9]{x9}
      \Vertex[x=4,y=0.75,LabelOut=true,Math=true,Lpos=180,L=x_{10}]{x10}
      \tikzset{VertexStyle/.append style = {fill=black}}
      \Vertex[x=4,y=4.5,LabelOut=true,Math=true,Lpos=180,L=u_1']{u3}
      \Vertex[x=4,y=3.5,LabelOut=true,Math=true,Lpos=180,L=u_2']{u4}
      \SetVertexNoLabel
      \Vertex[x=5,y=3,LabelOut=true,Math=true,Lpos=90]{b5}
      \Vertex[x=5,y=2,LabelOut=true,Math=true,Lpos=90]{b6}
      \Vertex[x=5,y=1,LabelOut=true,Math=true,Lpos=90]{b7}
      \Vertex[x=6,y=2.5,LabelOut=true,Math=true,Lpos=90]{u5}
      \Vertex[x=6,y=1.5,LabelOut=true,Math=true,Lpos=90]{u6}
      \Edge(b5)(u3)
      \Edge(b5)(u4)
      \Edge(b6)(x7)
      \Edge(b6)(x8)
      \Edge(b7)(x9)
      \Edge(b7)(x10)
      \Edge(u5)(b5)
      \Edge(u6)(b5)
      \Edge(u5)(b6)
      \Edge(u6)(b6)
      \Edge(u5)(b7)
      \Edge(u6)(b7)
      \Edge(u5)(u6)
    \end{tikzpicture} \\
    (c) $G_{6,4}$
  \end{subfigure}
  \caption{The intermediate gadgets used in the construction of $G_{10,4}$.}
  \label{gadgetexampleconstruction}
\end{figure}
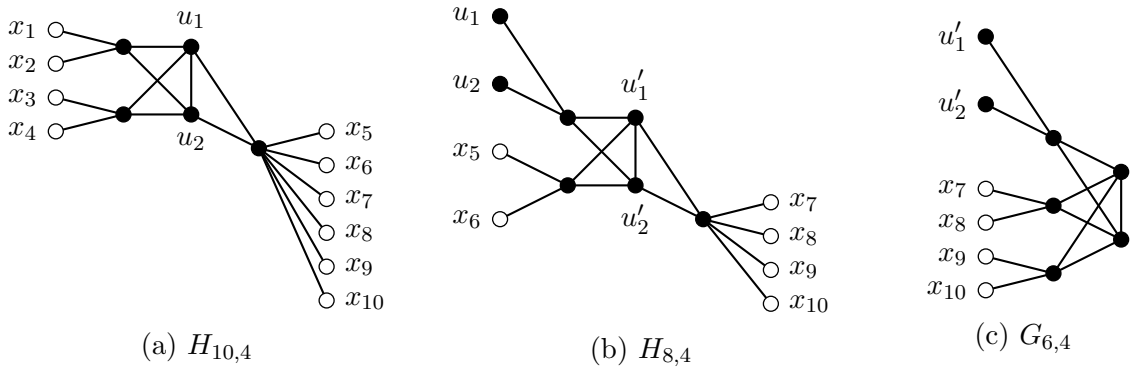

\begin{proposition}
  \label{colour_n4c_2'}
  For any fixed $k\geq 4$, the problem of deciding if a
  $(k-1)$-connected graph with maximal local edge-connectivity~$k$ is
  $3$-colourable is NP-complete.
\end{proposition}
\begin{proof}
  Let $G$ be a $(k-1)$-connected graph, and let $G'$ be the graph obtained by attaching a gadget $G_{d(v),k}$ to $G$ at $v$ for each vertex $v$ of degree at least $k+1$.
  It is not difficult to verify that $G'$ can be constructed in polynomial time and that every vertex of $G'$ has degree at most $k$, so $G'$ has maximal local edge-connectivity~$k$.  Moreover, for all distinct $i,j \in \seq{k-1}$, the vertices $\{b_i,b_j,u_1,u_2\}$ induce a diamond in $G'$, so the pick vertices $\{b_1,b_2,\dotsc,b_{k-1}\}$ of these diamonds must have the same colour in a $3$-colouring of $G'$.  Now, given a $3$-colouring of $G'$, we can $3$-colour $G$, where a vertex $v \in V(G)$ replaced by a gadget $G_{l,k}$ in $G'$ is given the colour shared by the vertices $\{b_1,b_2,\dotsc,b_{k-1}\}$ of $G_{l,k}$.  It is also straightforward to verify that if $G$ is $3$-colourable, then $G'$ is $3$-colourable.
  %We also observe that although the gadget $G_{d,k}$ was described recursively, it can be constructed in polynomial time, %(why?)
  %so we indeed have a polynomial-time reduction.

  It remains to show that $G'$ is $(k-1)$-connected.
  We may assume, by induction, that $G'$ is obtained from $G$ by attaching one gadget $G_{l,k}$.
  %that is, $G'=(G \cup G_{l,k}) - v$, for some $v \in V(G)$.
  Moreover, when $l > (k-2)(k-1)$, we can view the attachment of a gadget $G_{l,k}$ as a sequence of attachments of intermediate gadgets $H_{l_0,k},H_{l_1,k},H_{l_2,k},\dotsc,H_{l_{s-1},k}, G_{l_s,k}$ where
  %each $H_{i,k}$ has a single vertex of degree greater than $(k-2)(k-1)$, but the sequence $l,l_1,l_2,\dotsc,l_s$ is strictly decreasing %reduces the maximum degree by $(k-2)^2 - 2$,
  $l=l_0$ and $l_{i} = l_{i-1} - (k-2)^2 +2$ for $i \in \seq{s}$,
  and $l_s \le (k-2)(k-1)$.
  We need only to show that the attachment of 
  a gadget $G_{l,k}$, or of an intermediate gadget $H_{l,k}$, 
  preserves $(k-1)$-connectivity.

  Loosely speaking, we start by proving that 
  the gadget, or intermediate gadget, itself is sufficiently connected.
  Let $l > k \ge 4$.
If $l \le (k-2)(k-1)$, then set $J_{l,k} = G_{l,k}$, otherwise set $J_{l,k}=H_{l,k}$. 
  Let $K_l$ be a copy of the complete graph with vertex set $N_G(v)=\{x_1,x_2,\dotsc,x_l\}$. 
  We will prove that $J'_{l,k} = J_{l,k} \cup K_l$ is $(k-1)$-connected. 
  Let $(X,Z,Y)$ be a $j$-separation of $J'_{l,k}$, for some $j < k-1$, such that $Z$ is a minimal vertex cut. 
Set $U=\{u_1,u_2\}$  if $J_{l,k}=H_{l,k}$ and $U=\{u_1, \dots u_{k-a}\}$ if $J_{l,k}=G_{l,k}$. 
  Suppose $u_1 \in X$.
  Then $N(u_1) = \{b_1,b_2,\dotsc,b_{k-1}\} \cup \{u_2\}$ is contained in $X \cup Z$.   
  If $|U| \ge 3$ and there exists $u \in U \setminus \{u_1,u_2\}$ such that $u \in Y$, then $N(u) = \{b_1,b_2,\dotsc,b_{k-1}\} \subseteq Z$, contradicting $|Z| < k-1$. 
So $U \cup \{b_1, \dots, b_{k-1}\} \subseteq X \cup Z$ and thus  
  $Y \subseteq \{x_1,x_2,\dotsc,x_l\}$. 
For $i=1, \dots, k-1$, either $b_i \in Z$ or $B_i \subseteq Z$, so $|Z| \ge k-1$, a contradiction.
  Now we may assume that $u_1 \in Z$ and, by symmetry, $u_2 \in Z$.
  Since $Z$ is minimal, $u_1$ %there exist neighbours of $u_1$ 
  has a neighbour in $X$ and a neighbour in $Y$.  Suppose $b_i \in X$ and $b_j \in Y$ for $i,j \in \seq{k-1}$.
  In order for $Z$ to separate $b_i$ and $b_j$, we require %$U \cup B_i \cup B_j \subseteq Z$,
  $U \cup B_i \subseteq Z$ or $U \cup B_j \subseteq Z$,
  so $|Z| \ge k-1$, a contradiction.
  Thus $J'_{l,k}$ is $(k-1)$-connected.

  We now return to proving that $G'$ is $(k-1)$-connected.
  Suppose, towards a contradiction, that $G'$ is not $(k-1)$-connected. 
Let $(X,Z,Y)$ be a $j$-separation of $G'$ for some $j <k-1$, such that $Z$ is a minimal vertex cut.
  We denote by $G'_{l,k}$ the subgraph of $G_{l,k}$ obtained by deleting the vertices in common with $G$, namely $N_G(v)$.  Note that $V(G')$ is the disjoint union of $V(G-v)$ and $V(G'_{l,k})$.

  First, suppose that $Z \subseteq V(G-v)$.
  Since $G_{l,k}'$ is connected, and $Z \subseteq V(G-v)$, we deduce that, without loss of generality, $V(G_{l,k}') \subseteq Y$.
  It follows that
  $Z$ is a $j$-vertex cut that separates $X$ from $(Y \setminus V(G'_{l,k})) \cup \{v\}$ in $G$; a contradiction.  

  Now suppose that $Z \nsubseteq V(G-v)$.
  Moreover, suppose that  
  $X \cap V(G-v)$ and $Y \cap V(G-v)$ are both non-empty.
  Then $Z \cup V(G'_{l,k})$ separates $X  \cap V(G-v)$ from $Y \cap V(G-v)$ in $G'$, so $(Z \setminus V(G'_{l,k})) \cup \{v\}$  is a vertex cut of $G$ and since $Z \nsubseteq V(G)$, we have $|(Z \setminus V(G'_{l,k})) \cup \{v\}| \le j <k-1$, a contradiction to the fact that $G$ is $(k-1)$-connected.
Thus, either $X  \subseteq  V(G'_{l,k})$ or $Y  \subseteq  V(G'_{l,k})$.
  Assume without loss of generality that $X \subseteq V(G'_{l,k})$ and $V(G-v) \subseteq Y \cup Z$. 
Since $|N_G(v)| \ge k$, $Y \cap N_G(v) \neq \emptyset$. 
Hence $Z \cap (V(G'_{l,k}) \cup N_G(v))$ separates $Y \cap  (V(G'_{l,k}) \cup N_G(v))$ from $X$ in $J'_{l,k}$, a contradiction.
%
  %Now, if $Y \cap V(G'_{l,k}) \neq \emptyset$, then   $Z \cup V(G-v)$ separates $X$ from $Y \cap G_{l,k}'$ in $G'$.  It follows that   $Z \cap G'_{l,k}$ separates $X$ from $Y \cap G_{l,k}'$ in $J'_{l,k}$, in which case $J'_{l,k}$ has a vertex cut of fewer than $k-1$ elements; a contradiction.   So $X \subseteq V(G'_{l,k})$ and $Y \subseteq V(G)$, and   these sets are connected by a minimal $(k-1)$-edge cut,   so $|Z| \geq k-1$; a contradiction.  This completes the proof that $G'$ is $(k-1)$-connected. 
\end{proof}

\Cref{colour_n4c_2} now follows from \cref{mlcnpc3} and \cref{colour_n4c_2'}.
Note that \cref{colour_n4c_2'} rules out (unless P=NP) the possibility of a polynomial-time algorithm that computes the chromatic number (or finds an optimal colouring) for a graph with maximal local edge-connectivity~$k$, for $k \geq 4$.  However, there may exist a polynomial-time algorithm that, given such a graph, finds a $k$-colouring or determines that none exists.
Thus, a result in the style of \cref{thm:brooks3gen} that characterises when graphs with maximal local edge-connectivity~$4$ are $4$-colourable remains a possibility.

\section{Minimally $k$-connected graphs}

%\section{$k$-colouring a minimally $k$-connected graph is NP-hard}
\label{secnpc2}

In this section we
prove that %it is NP-hard to $k$-colour a minimally $k$-connected graph (\cref{m3c_npc}).
deciding if a minimally $k$-connected graph is $k$-colourable is NP-complete.
To do this, we perform a reduction from the following problem, where $k$ is a fixed integer at least three.
A hypergraph is \emph{$k$-uniform} if each hyperedge is of size $k$.
\\

%\noindent
%\textsc{$k$-colouring a $k$-uniform hypergraph} \\
%\textbf{Instance:} A $k$-uniform hypergraph $H$. \\
%\textbf{Question:} What is a $k$-colouring of $H$ for which no edge is monochromatic? \\
\noindent
\textsc{$k$-uniform hypergraph $k$-colourability} \\
\textbf{Instance:} A $k$-uniform hypergraph $H$. \\
\textbf{Question:} Is there a $k$-colouring of $H$ for which no edge is monochromatic? \\

%Dinur et al.\ \cite{Dinur2005} showed this problem is NP-hard (in fact, they prove %NP-hardness when $j$-colouring, for any $j \geq 3$, even when restricted to a $k$-uniform hypergraph that is $(k-1)$-colourable,
%a stronger result, but the above is sufficient for our needs).
The problem of deciding if a hypergraph is $2$-colourable is well known to be NP-complete~\cite{Lovasz1973}, and the search problem of finding a $k$-colouring for a $k$-uniform hypergraph, for $k \geq 3$, is shown in~\cite{Dinur2005} to be NP-hard, even when restricted to such hypergraphs that are $(k-1)$-colourable.
However, to the best of our knowledge, no proof that \textsc{$k$-uniform hypergraph $k$-colourability} is NP-complete has been published, so we provide one here for completeness.

\begin{proposition}
  The problem \textsc{$k$-uniform hypergraph $k$-colourability} is NP-complete for fixed $k \geq 3$.
\end{proposition}
\begin{proof}
  Let $V_1, V_2, \dotsc,V_k$ be $k$ disjoint sets each consisting of $k$ distinct vertices, and
  let $H_0$ be the $k$-uniform hypergraph with vertex set $V_1 \cup V_2 \cup \dotsm \cup V_k$ whose hyperedges consist of all $k$-element subsets of $V(H_0)$ not in $\{V_1, V_2, \dotsc, V_k\}$.
  Then, in a $k$-colouring of $H_0$, for each subset $X$ of $V(H_0)$ of size at least $k$, either $X$ is not monochromatic or $X$ is one of $V_1, V_2, \dotsc,V_k$.
  It follows that a $k$-colouring of $H_0$ is unique, up to a permutation of the colours: each $V_i$, for $i \in \seq{k}$, is monochromatic, and for distinct $i,j \in \seq{n}$, the colours given to $v_i \in V_i$ and $v_j \in V_j$ are distinct.

  We perform a reduction from \textsc{$k$-colourability}.  Let $G$ be a graph.  We construct a $k$-uniform hypergraph $H$ as follows. Start with the hypergraph on the vertex set $V(H_0) \cup V(G)$, where $V(H_0)$ and $V(G)$ are disjoint, and containing all the hyperedges of $H_0$.  For each edge $uv$ of $G$ and each $i \in \seq{k}$, introduce a hyperedge consisting of $u$, $v$, and $k-2$ vertices of $V_i$. %add $|V(G)|$ new verticesAdd a vertex ofFor each edge $uv \in E(G)$, 
  Each such hyperedge enforces that in a $k$-colouring of $H$, the vertices $u$ and $v$ do not both have colour~$i$.
%
  %Suppose $H$ is $k$-colourable.  Then, for each edge $uv$ of $G$, the vertices $u$ and $v$ are different colours in $G$, so $G$ is $k$-colourable.
  Thus, if $H$ is $k$-colourable, then $G$ is $k$-colourable.
  %(Then we can reduce $k$-colouring the following way: if we want to enforce that colour $i$ cannot appear on both $x$ and $y$, then we introduce a hyperedge containing $x$, $y$, and $k-2$ vertices of $V_i$. Repeating this for every $i$, we can force $x$ and $y$ to have different colours.)
  Now suppose that $\phi$ is a $k$-colouring of $G$.  Then, by assigning a vertex $v \in V(G)$ the colour~$\phi(v)$ in $H$, and colouring each vertex $v \in V(H_0)$ the colour~$i$ if $v \in V_i$, 
  we obtain a $k$-colouring of $H$.  This completes the proof.
\end{proof}

%Reminder:
%\begin{proposition}
  %For fixed $j\geq 3$, the problem of $k$-colouring a minimally $k$-connected
  %graph is NP-hard.
%\end{proposition}

\begin{proof}[Proof of \cref{m3c_npc}]
  %We perform a reduction from \textsc{$k$-colouring a $k$-uniform hypergraph}.
  We perform a reduction from \textsc{$k$-uniform hypergraph $k$-colourability}.
  Let $H$ be a $k$-uniform hypergraph with vertex set $\{v_1,v_2,\dotsc,v_{h}\}$.
  We will construct a minimally $k$-connected graph $G$ with $\{v_1,v_2,\dotsc,v_{h}\} \subseteq V(G)$.
  For each hyperedge $e=u_1u_2\dotsm u_k$, where $u_i \in \{v_1,v_2,\dotsc,v_{h}\}$ for $i \in \seq{k}$, let $P_e$ be the graph on $2k$ vertices that is the union of the complete graph $K_k$ on the vertices $\{k_1,k_2,\dotsc,k_k\}$, and $k$ vertex-disjoint edges $\{u_1k_1,u_2k_2,\dotsc,u_kk_k\}$.  For $k=3$, this graph is given in \cref{extensions-fig}.
  For each $l \in \seq{h}$, let $Q_l$ be the graph on $3k-1$ vertices obtained from the complete bipartite graph $K_{k,k-1}$ with $k$-element partite set $\{b_1,b_2,\dotsc,b_k\}$ by adding $k$ vertex-disjoint edges $b_iv_j$ where $j \equiv i + l \pmod h$ for each $i \in \seq{k}$. %where $u_i \notin V(K_{k,k-1})$.
  For $k=3$, this graph is given in \cref{extensions-fig}.
  Finally, we obtain $G$ from the union of the $i$ graphs $Q_i$, for each $i \in \seq{h}$, and the $|E(H)|$ graphs $P_e$, for each $e \in E(H)$.
  Note that a vertex $v_i$, for $i \in \seq{h}$, is common to $Q_{i-k},Q_{i-k+1},\dotsc,Q_{i-1}$ (with indices interpreted modulo $h$) and $P_e$ for any hyperedge $e$ containing $v_i$.

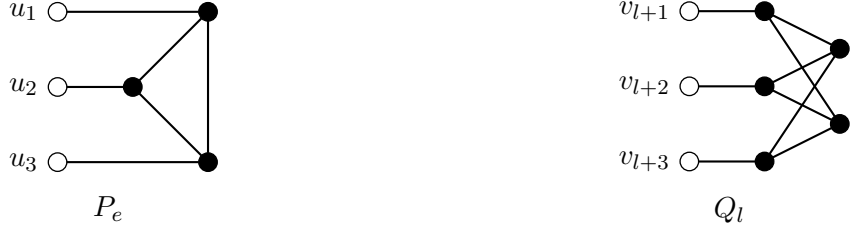
\begin{figure}
  \centering
  \begin{subfigure}{0.5\textwidth}
    \centering
    \begin{tikzpicture}[scale=1]
      \tikzset{VertexStyle/.append style = {draw,minimum height=7,minimum width=7}}
      \Vertex[x=0,y=2,LabelOut=true,L=$u_1$,Lpos=180]{a}
      \Vertex[x=0,y=1,LabelOut=true,L=$u_2$,Lpos=180]{b}
      \Vertex[x=0,y=0,LabelOut=true,L=$u_3$,Lpos=180]{c}
      \SetVertexNoLabel
      \tikzset{VertexStyle/.append style = {fill=black}}
      \Vertex[x=2,y=2,LabelOut=true,L=$k_1$,Lpos=180]{x1}
      \Vertex[x=1,y=1,LabelOut=true,L=$k_2$,Lpos=180]{x2}
      \Vertex[x=2,y=0,LabelOut=true,L=$k_3$,Lpos=180]{x3}
      \Edge(x1)(a)
      \Edge(x2)(b)
      \Edge(x3)(c)
      \Edge(x1)(x2)
      \Edge(x2)(x3)
      \Edge(x3)(x1)
    \end{tikzpicture}\\
     $P_e$
  \end{subfigure}%
  \begin{subfigure}{0.5\textwidth}
    \centering
    \begin{tikzpicture}[scale=1]
      \tikzset{VertexStyle/.append style = {draw,minimum height=7,minimum width=7}}
      \Vertex[x=0,y=2,LabelOut=true,L=$v_{l+1}$,Lpos=180]{a}
      \Vertex[x=0,y=1,LabelOut=true,L=$v_{l+2}$,Lpos=180]{b}
      \Vertex[x=0,y=0,LabelOut=true,L=$v_{l+3}$,Lpos=180]{c}
      \SetVertexNoLabel
      \tikzset{VertexStyle/.append style = {fill=black}}
      \Vertex[x=1,y=2]{x1}
      \Vertex[x=1,y=1]{x2}
      \Vertex[x=1,y=0]{x3}
      \Vertex[x=2,y=1.5]{y1}
      \Vertex[x=2,y=0.5]{y2}
      \Edge(x1)(a)
      \Edge(x2)(b)
      \Edge(x3)(c)
      \Edge(y1)(x1)
      \Edge(y1)(x2)
      \Edge(y1)(x3)
      \Edge(y2)(x1)
      \Edge(y2)(x2)
      \Edge(y2)(x3)
    \end{tikzpicture}\\
     $Q_l$
  \end{subfigure}
  \caption{Gadgets used in the proof of \cref{m3c_npc}.}
  \label{extensions-fig}
\end{figure}

  Suppose we have a $k$-colouring for $G$.  Then, since each vertex of a $K_k$ subgraph is coloured a different colour, each set of vertices $\{u_1,u_2,\dotsc,u_k\}$ corresponding to a hyperedge $e$ is not monochromatic.  So the vertex colouring of the graph $G$ gives us a colouring of the hypergraph $H$ where no hyperedge is monochromatic.

  Now suppose we have a colouring $\phi$ of $H$ where no hyperedge is monochromatic.  Starting from the colouring on $\{v_1, v_2, \dotsc, v_{h}\}$ given by %the colouring of $H$, %it is easily verified that
  $\phi$,
  we can extend this to a colouring of $G$ as follows. 
  Consider a $Q_l$ subgraph.  For each $v \in \{v_{l+1},v_{l+2},\dotsc,v_{l+k}\}$, if $\phi(v) \neq 1$, we assign the vertex adjacent to $v$ in $Q_l$ colour~$1$; otherwise, it is assigned colour~$2$.  The remaining $k-1$ vertices of $Q_l$ can then be assigned colour~$3$.
  Now consider a $P_e$ subgraph. Let $U = \{u_1,u_2,\dotsc,u_k\}$, let $C$ be the set of colours $\phi(U)$, and let $\sigma$ be a permutation of $C$ with no fixed points; such a permutation exists since no hyperedge of $H$ is monochromatic, so $|C| > 1$.
  For each $c \in C$, pick a vertex $u \in U$ with $\phi(u) = c$, and colour the vertex adjacent to $u$ in $P_e$ the colour $\sigma(c)$.
  Now each of the remaining $k-|C|$ uncoloured vertices can be assigned one of the $k-|C|$ unused colours arbitrarily.
  So $G$ is $k$-colourable if and only if $H$ is $k$-colourable.

  For every edge $xy$ of $G$, at least one of $x$ or $y$ has degree $k$, so $G \ba xy$ is at most $(k-1)$-connected.
  Moreover, it is not difficult to see there are at least $k$ internally disjoint paths between any pair of vertices, so $G$ is $k$-connected.  Hence $G$ is minimally $k$-connected, as required.
\end{proof}

%\section{Graphs with no $(k+1)$-connected subgraph}

%In this short section we show that \textsc{$3$-colourability}, when restricted to graphs in $\mathcal{C}_4^2$, is NP-complete.

%\label{secnpc3}
%\begin{proof}[Proof of \cref{npcc3}]
  %We perform a reduction from \textsc{$3$-colourability}.
  %Let $G$ be an instance of \textsc{$3$-colourability}.
  %If $G$ is not $2$-connected, we can decompose $G$ into polynomially many blocks, and $G$ is $3$-colourable if and only if each of the blocks is $3$-colourable.
  %So we can assume that $G$ is $2$-connected.
%%
  %Let $D$ be a diamond with pick vertices $u$ and $v$.
  %Let $D'$ be the graph with vertex set $V(D) \cup \{z\}$ and edge set $E(D) \cup \{vz\}$ (this graph is sometimes called a \emph{dart}).
  %For each edge $xy$ of $G$, remove $xy$, take a copy of $D'$, and identify $u$ with $x$ and $y$ with $z$.
  %Label the resulting graph $G'$.
  %It is easy to verify that
  %$G'$ is $2$-connected and has no $3$-connected subgraph, and 
  %$G$ is $3$-colourable if and only if $G'$ is $3$-colourable.
%\end{proof}

\section{Graphs with a bounded number of vertices of degree more than~$k$}
\label{secbigvs}
%%%%%%%%%%%%%%%%%%%%%%%%%%%%%%%%%

In this section we prove \cref{fewbigvs}.  
The proof of this result relies on a generalisation of Brooks' theorem established independently by Borodin~\cite{Bor77} and by Erd\H{o}s, Rubin and Taylor~\cite{ERT79}.

A \emph{list assignment} for a graph $G$ is a function $L$ that
associates to every vertex $v\in V(G)$ a set $L(v)$ of integers that
are called the \emph{colours associated with $v$}.  A {\it
  degree-list-assignment} of a graph $G$ is a list assignment $L$ such
that $|L(v)|\geq d_G(v)$ for every $v\in V(G)$.  An
\emph{$L$-colouring} of $G$ is a function $c$ from $V(G)$ such that,
for all $v\in V(G)$, we have $c(v) \in L(v)$, and, for all edges $uv$,
we have $c(u)\neq c(v)$.  A graph $G$ is \emph{$L$-colourable} if it admits at
least one $L$-colouring.  A graph $G$ is \emph{degree-choosable} if $G$
is $L$-colourable for any degree-list-assignment $L$.  A 
graph is a \emph{Gallai tree} if it is connected and each of its
blocks is either a complete graph or an odd cycle.

\begin{theorem}[Borodin~\cite{Bor77}, Erd\H{o}s, Rubin and
   Taylor~\cite{ERT79}]
   \label{gallaitree}
   Let $G$ be a connected graph.  Then $G$ is degree-choosable if and
   only if $G$ is not a Gallai tree.  Moreover, if $G$ is not a Gallai
   tree, then there is an $O(m)$-time algorithm that, given a
   degree-list-assignment $L$, finds an $L$-colouring.
\end{theorem}

We need now to study $L$-colourings of Gallai trees.  Let $G$ be a
Gallai tree together with a list assignment $L$.  Suppose that $G$ has
a cut-vertex and consider a leaf block $B$ attaching at $v$.  We say
that $L$ is \emph{$B$-uniform} if the list $L(u)$ is the same for all
$u \in V(B) \setminus \{v\}$ and satisfies $|L(u)|=d(u)$.  When $L$ is
$B$-uniform, we define the list assignment $L_{\overline{B}}$ of
$G-V(B - \{v\})$ as follows: for all $w\neq v$,
$L_{\overline{B}}(w) = L(w)$ and
$L_{\overline{B}}(v) = L(v) \setminus L(u)$ for some, and thus any,
$u\in V(B)\setminus \{v\}$.

\begin{lemma}
  \label{l:leafBlo}
  If a Gallai tree $G$ has a cut-vertex $v$, a leaf block $B$
  attaching at $v$, and a list assignment $L$ such that $L$ is
  $B$-uniform, then $G$ is $L$-colourable if and only if
  $G - V(B - \{v\})$ is $L_{\overline{B}}$-colourable.
\end{lemma}

\begin{proof}
  We deal only with the case when $B$ is an odd cycle (the case when
  $B$ is a complete graph is similar). Up to a relabelling of the colours, we may
  assume that every vertex of $B - \{v\}$ is assigned the list
  $\{1, 2\}$.

  Suppose that $G - V(B - \{v\})$ is
  $L_{\overline{B}}$-colourable.  In the colouring of
  $G - V(B -\{v\})$, the colours $\{1, 2\}$ are not used for $v$
  (from the definition of $L_{\overline{B}}$), so they can be used to
  colour $B-\{v\}$, showing that  $G$ is $L$-colourable.

  Suppose conversely that $G$ is $L$-colourable.  We note that
  $B - \{v\}$ is a path of odd length and, in any $L$-colouring,
  its ends must receive colours 1 and 2, because of the parity. It
  follows that $v$ is not coloured with 1 or 2.  Therefore, the
  restriction of the colouring to $G - V(B - \{v\})$ is an
  $L_{\overline{B}}$-colouring, showing that $G - V(B - \{v\})$
  is $L_{\overline{B}}$-colourable. 
\end{proof}

When $G$ is a graph together with a degree-list-assignment $L$, we say a
vertex has a \emph{long} list $L(v)$ when $|L(v)|>d(v)$.  

\begin{lemma}
  \label{l:bigL}
  There is an $O(m)$-time algorithm whose input is a connected graph $G$
  together with a degree-list-assignment $L$ such that at least one
  vertex has a long list, and whose output is an $L$-colouring of $G$. 
\end{lemma}

\begin{proof}
  In time $O(m)$, a vertex $v$ whose list is long can be identified.
  The algorithm then runs a search of the graph (a depth-first search, for instance)
  starting at $v$.  This gives a linear ordering of the vertices
  starting at $v$: $v=v_1 < v_2 < \cdots < v_n$, such that, for every
  $i \in \{2, 3, \dotsc, n\}$, the vertex~$v_i$ has at least one neighbour
  $v_j$ with $j<i$.  The greedy colouring algorithm starting at $v_n$
  then yields an $L$-colouring of $G$.
\end{proof}

\begin{proposition}
\label{l:colourG}
  There is an $O(m)$-time algorithm whose input is a Gallai tree $G$
  together with a degree-list-assignment $L$, and whose output is an
  $L$-colouring of $G$ or a certificate that no such colouring exists.
\end{proposition}

\begin{proof}
  The algorithm first checks whether one of the lists is long, and if so runs
  the algorithm from \cref{l:bigL}.  Otherwise the classical $O(m)$-time
  algorithm of Tarjan~\cite{tarjan:dfs} finds the block
  decomposition of $G$.  

  Loop step: If $G$ is not a clique or an odd cycle, then it has a cut-vertex $v$
  and a leaf block $B$ attaching at $v$.  The algorithm checks whether $L$
  is $B$-uniform (which is easy in time $O(|V(B)|)$), and if so, as in the proof of \cref{l:leafBlo}, colours the
  vertices of $B - \{v\}$, removes them, updates the list $L(v)$, and
  repeats the loop step  again.  

  If $B$ is not uniform, then the algorithm identifies in $B - \{v\}$ two adjacent
  vertices $u, u'$ with different lists.  So,
  up to swapping $u$ and $u'$, there is a colour $c$ in $L(u)$ that is not present in
  $L(u')$.  Then the algorithm gives colour $c$ to $u$, removes $u$
  from $G$, and removes colour $c$ from the lists of all neighbours of $u$.
  The resulting graph is a connected graph together with a
  degree-list-assignment, and the list of $u'$ is long.  Therefore,
  we may complete the colouring by \cref{l:bigL}.   

  Hence, we may assume that the algorithm repeats the loop step
  until the removal of leaf blocks finally leads to a clique or an odd
  cycle.  Then, if all lists of vertices are equal, obviously no
  colouring exists, and the sequence of calls to \cref{l:leafBlo}
  certifies that $G$ has no $L$-colouring.  Otherwise, a colouring can
  be found %exactly as in the paragraph above.
  by \cref{l:bigL}.
\end{proof}

\begin{proof}[Proof of \cref{fewbigvs}]
Let $X$ be the set of $p$ vertices of degree more than $k$ in $G$.
We guess what could be the colouring on those vertices. There are at most $k^p$ possibilities.

For each, we check whether it can be extended to a $k$-colouring of the whole graph.
To do so we consider $H= G-X$, and for every vertex $v\in V(G)\setminus X$, we use the list assignment $L(v)$ given by the list of colours in $\seq{k}$ that are not used on a neighbour of $v$ in $X$.
Clearly, $|L(v)|\geq k - |N(v)\cap X| \geq d_H(x)$, so we have a degree-list-assignment.

Next we find the connected components of $H$ in $O(n+m)$.
Then for each component $C$, we check if $C$ is a Gallai tree or not.  If not, then we use the
$O(m)$ algorithm of \cref{gallaitree} to $L$-colour $C$.  If it is
a Gallai tree, then we rely on \cref{l:colourG}.  

The running time of the algorithm described above is $k^pO(n+m)$. If $k > p$, then we may assume, without loss of generality, that only the first $p$ colours are used on the $p$ vertices of degree more than $k$.  Therefore, we have to try only $p^p$ possibilities for colouring these vertices. Thus, we obtain an algorithm that runs in time $\min\{k^p,p^p\} \cdot O(n+m)$.
\end{proof}

\Cref{fewbigvs} immediately implies a fixed-parameter tractability result.
\begin{corollary}
  The problem \textsc{$k$-colourability}, when parameterized by the number of vertices of degree more than $k$, is FPT.
\end{corollary}

Let us consider now the problem restricted to the case when $G$ is
planar. Then only the $k=3$ case makes sense: for $k\le 2$, the
problem is polynomial-time solvable, while for $k\ge 4$ the colouring always
exists by the Four Colour Theorem. For $k=3$, \cref{fewbigvs}
gives an algorithm with running time $3^p\cdot O(n+m)$. On general
graphs, this is essentially best possible, in the following sense. The
{\em Exponential-Time Hypothesis (ETH)}, formulated by Impagliazzo,
Paturi, and Zane \cite{MR1894519}, implies that $n$-variable
\textsc{3-SAT} cannot be solved in time $2^{o(n)}$. It is known that ETH further implies that \textsc{$3$-colourability} on an $n$-vertex graph cannot be solved in time $2^{o(n)}$ \cite{survey-eth-beatcs}. It follows that in the algorithm given by \cref{fewbigvs} for  \textsc{$3$-colourability}, the exponential dependence on $p$ cannot be improved to $2^{o(p)}$: as $p$ is at most the number of vertices, such an algorithm could be used to solve \textsc{$3$-colourability} in time $2^{o(n)}$ on any graph.
\begin{corollary}
Assuming ETH, there is no $2^{o(p)}\cdot n^{O(1)}$ time algorithm for \textsc{$3$-colourability}, where $p$ is the number of vertices with degree more than 3.
\end{corollary}
However, on planar graphs we can do substantially better. There are
several examples in the parameterized-algorithms literature
\cite{DBLP:conf/soda/ChitnisHM14,DBLP:journals/cj/DemaineH08,DBLP:conf/icalp/KleinM12,DBLP:conf/soda/KleinM14,DBLP:conf/fsttcs/LokshtanovSW12,DBLP:conf/focs/PilipczukPSL14}
where significantly better algorithms are known when the problem is
restricted to planar graphs, and, in particular, a square root appears in
the running time.  In most cases, the square root comes from the use
of the Excluded Grid Theorem for planar graphs, stating that if a
planar graph has treewidth $w$, then it contains an $\Omega(w)\times
\Omega(w)$ grid minor. Often this result is invoked not on the input
graph itself, but on some other graph derived from it in a nontrivial
way. This is also the case with this problem.

\begin{theorem}\label{cor:planar}
  Let $G$ be a planar graph with at most $p$ vertices of degree more than
  $3$.  There is a $2^{O(\sqrt{p})}(n+m)$-time algorithm for $3$-colouring $G$,
  or determining no such colouring exists.
\end{theorem}
\begin{proof}
  Let $X$ be the set of $p$ vertices of degree more than $3$ in $G$.
If a component $C$ of $G-X$ is not a Gallai tree, then,
by \cref{gallaitree}, we can extend a colouring of $G \setminus C$ to a colouring of $G$ in linear time (similar to the proof of \cref{fewbigvs}).  
Thus, we may assume that each component of $G-X$ is a Gallai tree.
%A planar Gallai tree has no cliques of size more than $4$, thus it has treewidth at most $3$.
It is well known that the treewidth of a graph is the maximum treewidth of one of its blocks (see, for example, \cite{lozin2004band}).
Since a planar Gallai tree has no cliques of size more than $4$, and an odd cycle has treewidth $2$, a planar Gallai tree has treewidth at most $3$.
%To see this, we can, firstly, compute a tree decomposition with a bag for each block, where blocks that share a cut vertex have an edge between them in the tree.  Then, for each bag $B$ corresponding to an odd cycle of length at least 5, we can replace $B$ by a path of bags $P_1, P_2, \dotsc, P_n$ each of size three, where a bag $N$ adjacent to $B$ is now adjacent to one of the bags $P_i$ containing the cut vertex common to $B$ and $N$, for $i \in \{1,\dotsc,n\}$.
Therefore, the deletion of $X$ from $G$ reduces the treewidth of the resulting graph to a constant.
Let $w$ be the treewidth of $G$.
Since $G$ is planar, it contains a $\Omega(w) \times \Omega(w)$ grid minor,
so $\Omega(w^2)$ vertices need to be deleted in order to reduce the treewidth of $G$ to a constant. This implies that $p=|X|=\Omega(w^2)$, or in other words, $G$ has treewidth $O(\sqrt{p})$.
Therefore, after computing a constant-factor approximation of the tree decomposition (using, for example, the algorithm of Bodlaender et al.~\cite{DBLP:conf/focs/BodlaenderDDFLP13} or Kammer and Tholey \cite{DBLP:conf/soda/KammerT12}), we can use a standard $3$-colouring on the tree decomposition to solve the problem in time $2^{O(\sqrt{p})} \cdot n$.
\end{proof}
It is known that, assuming ETH, \textsc{3-colourability} cannot be
solved in time $2^{o(\sqrt{n})}$ on planar graphs
\cite{survey-eth-beatcs}. This implies that the $2^{O(\sqrt{p})}$
factor in \cref{cor:planar} is best possible: assuming ETH, it
cannot be replaced by $2^{o(\sqrt{p})}$.

\bibliographystyle{abbrv}
\bibliography{library}
\end{document}